\documentclass[11pt]{article}
\usepackage[]{amsmath,amssymb,amsfonts,latexsym,amsthm,enumerate,fullpage}
\newcommand{\remove}[1]{}

\newcommand{\calB}{{\cal B}}

\newcommand{\calO}{{\cal O}}

\def\F{\mathbb{F}}
\def\Q{\mathbb{Q}}
\def\R{\mathbb{R}}
\def\C{\mathbb{C}}
\def\P{\mathbb{P}}
\def\N{\mathbb{N}}
\def\Z{\mathbb{Z}}

\newtheorem{thm}{Theorem}[section]
\newtheorem{definition}[thm]{Definition}

\newtheorem{remark}[thm]{Remark}
\newtheorem{cor}[thm]{Corollary}

\newtheorem{lemma}[thm]{Lemma}
\newtheorem{proposition}[thm]{Proposition}

\newcommand{\tnote}[1]{ \marginpar{\tiny\bf
             \begin{minipage}[t]{0.5in}
               \raggedright #1
            \end{minipage}}}

\begin{document}
\title{Isoperimetric  Inequalities for Ramanujan Complexes \\
       and Topological Expanders\footnote{The results of this paper were announced in~\cite{KaufmanKazhdanLubotzky}.}}%\\
%       DRAFT}
\author{Tali Kaufman \thanks{Bar-Ilan University, ISRAEL. Email: \texttt{kaufmant@mit.edu}.
Research supported in part by the Alon Fellowship, IRG, ERC and BSF.}
 \and David Kazhdan \thanks{Hebrew University, ISRAEL. Email: \texttt{kazhdan.david@gmail.com}.
Research supported in part by the NSF, BSF and ERC.}
 \and Alexander Lubotzky \thanks{Hebrew University, ISRAEL. Email: \texttt{alexlub@math.huji.ac.il}.
Research supported in part by the ERC, NSF and ISF.}}
%\author{Tali Kaufman, David Kazhdan, Alexander Lubotzky}
\maketitle
\begin{abstract}
Expander graphs have been intensively studied in the last four decades. In recent years a high dimensional theory of expanders has emerged, and several variants have been studied. Among them stand out {\em coboundary expansion} and {\em topological expansion}. It is known that for every $d$ there are {\em unbounded degree} simplicial complexes of dimension $d$ with these properties. However, a major open problem, formulated by Gromov~\cite{Gromov}, is whether {\em bounded degree} high dimensional expanders exist for $d \geq 2$.

We present an explicit construction of {\em bounded degree} complexes of dimension $d=2$ which are {\em topological expanders}, thus answering Gromov's question in the affirmative. Conditional on a conjecture of Serre on the congruence subgroup property, infinite sub-family of these give also a family of {\em bounded degree coboundary expanders}.

The main technical tools are new isoperimetric inequalities for Ramanujan Complexes. We prove {\em linear size} bounds on $\F_2$ systolic invariants of these complexes, which seem to be the first {\em linear} $\F_2$ systolic bounds. The expansion results are deduced from these isoperimetric inequalities.

\end{abstract}

\newpage

\tableofcontents

\newpage

\section{Introduction}
%\tnote{add discussion: about Garland; about first linear systolic results}
A classical result of Boros and F\"{u}redi~\cite{BorosFuredi} (for $d=2$) and B\'{a}r\'{a}ny~\cite{Barany} (for general $d \geq 2$) asserts that there exists $\epsilon_d >0$ such that given any set of $n$ points in $\R^d$, there exists $z \in \R^d$ which is contained in at least $\epsilon_d$-fraction of the ${n \choose d+1}$ simplicies determined by the set. Gromov~\cite{Gromov} changed the perspective of this result by strengthening and generalizing it in the following way.

\begin{definition}\label{def:geometric-top-overlapping} Let $X$ be a $d$-dimensional pure simplicial complex with a set $X(0)$ of vertices and denote by $X(d)$ the set of $d$-dimensional faces.
\begin{enumerate}
\item\label{item:def-geometric-top-overlapping-first-item} We say that $X$ has the {\em $\epsilon$-geometric overlapping property}, for some $0 < \epsilon \in \R$, if for every $f:X(0) \rightarrow \R^d$, there exists a point $z \in \R^d$ which is covered by at least $\epsilon$-fraction of the images of the faces in $X(d)$ under $\tilde{f}$. Here, $\tilde{f}$ is the (unique) affine extension of $f$.
\item\label{item:def-geometric-top-overlapping-second-item} We say that $X$ has the {\em $\epsilon$-topological overlapping property}, if the same conclusion holds for every continuous extension  $\tilde{f}:X \rightarrow \R^d$ of $f$.
\item\label{item:def-geometric-top-overlapping-third-item} A family of $d$-dimensional simplicial complexes is a geometric (resp. topological) expander if all of them have the $\epsilon$-geometric (resp. topological) overlapping property for the same $\epsilon > 0$.
\end{enumerate}
\end{definition}
%\tnote{check def}

Barany's theorem is, therefore, the statement that, for every $d$, $\Delta_n^{(d)}$ - the complete $d$-dimensional simplicial complex on $n$ vertices of dimension $d$ are geometric expanders. Gromov proved the remarkable result, saying that they are also topological expanders (The reader is encouraged to think about the case $d=2$ to see how non-trivial is this result and even somewhat counter intuitive!) Moreover, he went ahead and showed that various other families of simplicial complexes (of fixed dimension $d$) have the topological overlapping property, e.g., spherical buildings (see~\cite{Gromov},~\cite{LubotzkyMeshulamMozes}).

All the examples shown by Gromov are of simplicial complexes of unbounded degrees, i.e., the number of $d$-faces containing a fixed vertex (or even the number of $d$ faces containing a $(d-1)$-face) is unbounded along the family. He suggested~\cite[p.422]{Gromov} that $2$-dimensional Ramanujan complexes (see below) coming from a fixed local non-archimedean field $F$, form a family of bounded degree topological expanders "if such at all exist..." in his words. He showed that they have a weaker property, namely, the conclusion of Definition~\ref{def:geometric-top-overlapping}(\ref{item:def-geometric-top-overlapping-second-item}) holds if $\tilde{f}$ is $k$ to $1$ on faces, for some fixed $k$.  Even the question of existence of families which are bounded degree geometric expanders for general $d$ was left open in~\cite{Gromov}.

The later question, i.e. the question of existence of families of simplicial complexes of bounded degree with the {\em geometric overlapping property} was resolved by Fox, Gromov, Lafforgue, Naor, and Pach~\cite{FGLNP} in a satisfactory way (and in several ways). They showed, for example, the following result.

\begin{thm}\label{thm:FGLNP}
Let $d \geq 2$ and fix a sufficiently large prime power $q$: If $F$ is a local non-archimedean field of residual degree $q$, then the Ramanujan complexes quotients of the Bruhat-Tits building associated with $PGL_{d+1}(F)$ are $d$-dimensional bounded degree geometric expanders.
\end{thm}

However, they also left open the question of existence of bounded degree topological expanders.

In~\cite{LubotzkyMeshulam} a model of random $2$-dimensional complexes of bounded {\em edge} degree is presented. These complexes are shown to be topological expanders, but they, also, have unbounded (vertex) degree.

In this paper we show, for the first time, the existence of bounded degree $2$-dimensional topological expanders. We fell short from proving that Ramanujan complexes of dimension $2$ are themselves topological expanders, but we prove:

\begin{thm}\label{thm-main} Fix a sufficiently large prime power $q$ and let $F=\F_q((t))$. Let $\{X_{a}\}_{a \in A}$ be the family of $3$-dimensional non-partite Ramanujan complexes obtained from the Bruhat-Tits building associated with $PGL_{4}(F)$~\cite{LSV2}. For each such $X_{a}$, let $Y_{a}= X_{a}^{(2)}$ - the $2$-skeleton of $X_{a}$. Then, the family of $2$-dimensional simplicial complexes $\{Y_{a}\}_{a \in A}$ is an infinite family
of bounded degree topological expanders. %(of degree at most $O(q^6)$).
\end{thm}

Before elaborating on the method of proof, let us start by relating the above mentioned results to the notion of {\em coboundary expanders} as (essentially) been defined by Linial and Meshulam~\cite{LinialMeshulam} in a completely different context. Their motivation was to generalize to complexes the Erdos-R\`{e}yni theory of random graphs.

To introduce this and to present the main technical results of this paper we need few notations: Let $X$ be a pure $d$-dimensional simplicial complex (i.e., every maximal simplex is $d$-dimensional). For $i\leq d$, let $X(i)$ be the set of $i$-cells of $X$ and for $\sigma \in X(i)$, denote $c(\sigma)=|\{ \tau \in X(d) | \sigma \subseteq \tau\}|$ and $w(\sigma)=\frac{c(\sigma)}{{d+1 \choose i+1}|X(d)|}$. This weight function on $X(i)$ defines a "norm" on $C^i=C^i(X,\F_2)=\{f:X(i) \rightarrow \F_2\}$ by $||\alpha||=\sum_{\sigma \in \alpha} w(\sigma)$, where $\alpha \in C^i$ is considered also as the subset $\{\sigma \in X(i) \mbox{ } | \mbox{ } \alpha(\sigma) \neq 0\}$ of $X(i)$. As usual $\delta=\delta_i:C^i \rightarrow C^{i+1}$ is the coboundary map $\delta(\alpha)(\sigma)=\sum_{\tau \subseteq \sigma, |\tau|=i}\alpha(\tau)$ for $\alpha \in C^i$ and $\sigma \in X(i+1)$. As $\delta_i \circ \delta_{i-1} = 0$, $B^i  \subseteq Z^i$ where $B^i=\mbox{Im}(\delta_{i-1})$ (resp. $Z^i=\mbox{Ker}(\delta_i)$) is the space of $i$-coboundaries (resp. $i$-cocycles).

\begin{definition}\label{def-coboundary-cocycle-exp}.

\begin{enumerate}
\item($\F_2$-coboundary expansion) For $i=0,1,\cdots, d-1$, denote
$$\epsilon_i(X) := \mbox{min}\{\frac{||\delta_i \alpha  ||}{||[\alpha] ||} \mbox{ } |  \mbox{ }\alpha \in C^i \setminus B^i \}$$
When $[\alpha] = \alpha+B^i$ and $||[\alpha]|| = \mbox{min}\{||\gamma|| \mbox{ } |  \mbox{ } \gamma \in [\alpha]\}$.

\item($\F_2$-cocycle expansion) For $i=0,1,\cdots, d-1$, denote
$$\tilde{\epsilon}_i(X) := \mbox{min}\{\frac{||\delta_i \alpha  ||}{||\{\alpha\} ||} \mbox{ } |  \mbox{ }\alpha \in C^i \setminus Z^i \}$$
When $\{\alpha\} = \alpha+Z^i$ and $||\{\alpha\}|| = \mbox{min}\{||\gamma|| \mbox{ } |  \mbox{ } \gamma \in \{\alpha\}\}$.

\item(cofilling constant) The $i$-th cofilling constant of $X$, $0 \leq i \leq d-1$, is
$$\mu_i(X) = :\mbox{max}_{0 \neq \beta \in B^{i+1}}\{\frac{1}{||\beta||} \mbox{min}_{\alpha \in C^i, \delta \alpha = \beta} ||\alpha||\}$$.
\end{enumerate}
\end{definition}

If $\{X_j\}_{j \in J}$ is a family of $d$-dimensional simplicial complexes with $\epsilon_i(X_j) \geq \epsilon$ (resp. $\tilde{\epsilon_i}(X_j) \geq \epsilon$) for some $\epsilon > 0$ and every $0 \leq i \leq d-1$ and $j\in J$, we say that this is a family of {\em coboundary } (resp. {\em cocycle}) {\em expanders}. Note that $\{X_j\}_{j \in J}$ is a family of cocycle expanders iff there exists $M \in \R$ such that $\mu_i(X_j) \leq M$ for every $i=0, \cdots, d-1$ and $j \in J$.

As $B^0 =\{\mathbf{0},\mathbf{1}\}$, one easily checks that $\epsilon_0=\epsilon_0(X)$ is the normalized Cheeger constant of the $1$-skeleton of $X$, so the $\epsilon_i$'s deserve to be considered as a generalization of the notion of expansion of graphs. Meshulam and Wallach~\cite{MeshulamWallach} on one hand and Gromov~\cite{Gromov} on the other hand showed that $\Delta_n^{(d)}$ form a family of coboundary expanders. But, also in these works the existence of coboundary expanders of bounded degree remained open.

It is easy to see that $\mu_i = \frac{1}{\tilde{\epsilon_i}}$ and that if $\epsilon_i(X) > 0$ then $H^i(X,\F_2)=0$, in which case $\mu_i = \frac{1}{\epsilon_i}$. A family of coboundary expanders is therefore a family with bounded filling norms, but not vise versa. Also, for $d$-dimensional coboundary expanders the $\F_2$ cohomology vanish for every $i < d$.

Ramanujan complexes are in general not coboundary expanders. In fact we will show:

\begin{proposition}\label{prop-non-trivial-first-second-cohomology}
Let $d \geq 2$, $F=\F_q((t))$ and $B=\tilde{A}_d(F)$ the Bruhat-Tits building associated with $PGL_{d+1}(F)$. Then, $B$ has infinitely many quotients $X$ which are Ramanujan complexes with both $H^1(X,\F_2)$ and $H^2(X,\F_2)$ non-zero.
%Let $F$ be a local field and $\calB$ the Bruhat-Tits building of $PGL_d(F)$, $d \geq 3$ and $\Gamma$ a cocompact lattice in $PGL_d(F)$. Then $\Gamma$ has a finite index subgroups $\Gamma_1$ with $H^1(\Gamma_1,\F_2) \neq 0$ and $H^2(\Gamma_1,\F_2) \neq 0$.
\end{proposition}

Proposition~\ref{prop-non-trivial-first-second-cohomology} should be compared with a well known result of Garland~\cite{Garland} asserting that for such $X$, the real $i$-cohomology always vanish, i.e., $H^i(X,\R)=0$, for every $i < d$.

%The proof implies that there are even Ramanujan complexes $X_1$ with $H^1(X_1,\F_2) \neq 0$ and $H^2(X_1,\F_2) \neq 0$, and hence they are not coboundary expanders.
%\tnote{check this statement}

A deep result of Gromov~\cite{Gromov} asserts that coboundary expanders are topological expanders. While there are several methods to prove geometric overlapping (\cite{FGLNP},\cite{Parzanchevsky}), this is essentially the only known method to prove topological overlapping. As Proposition~\ref{prop-non-trivial-first-second-cohomology} shows, Ramanujan complexes, in general, are {\em not} coboundary expanders (see further discussion in Section~\ref{subsection-Bruhat-Tits buildings} and in Section~\ref{section-Serre-conj-implications}). By the same reason, the $Y_{a}$'s of Theorem~\ref{thm-main} are not coboundary expanders. We are still able to show that they are topological expanders due to two reasons. First, the following result~\cite{KaufmanWagner} which extends Gromov's criterion to the cases where the cohomology does not necessarily vanish. Unfortunately, we need some more notations.

\begin{definition}\label{def-systole} For a finite $d$-dimensional simplicial complex $X$ and $1 \leq i \leq d-1$ denote
$$\mbox{syst}^i(X) = \mbox{min}\{||\alpha|| \mbox{ } | \mbox{ } \alpha \in Z^i(X,\F_2)\setminus B^i(X,\F_2) \}.$$
(Write $\mbox{syst}^i(X)=\infty$ if $H^i(X,\F_2)=0$.) This is the {\em $i$-cohomological systole} of $X$ over $\F_2$.
\end{definition}

\begin{thm}\label{thm-Gromov-criteria-for-top-exp}
Given $0 < \mu, \eta \in \R$, and $d \in \N$, there exists $c=c(d,\mu,\eta)>0$ such that if $X$ is a finite pure simplicial complex of dimension $d$ satisfying:
\begin{enumerate}
\item\label{item:thm:Gromov-systolic:first} For every $0 \leq i \leq d-1$, $\mu_i(X) \leq \mu$.
\item\label{item:thm:Gromov-systolic:second} For every $0 \leq i \leq d-1$, $\mbox{syst}^i(X) \geq \eta$.
\end{enumerate}
Then $X$ has the $c$-topological overlapping property.
\end{thm}
In different words, cocycle expanders with large systole are topological expanders.

A proof of Theorem~\ref{thm-Gromov-criteria-for-top-exp} is given in~\cite{KaufmanWagner}.

Thus, to prove Theorem~\ref{thm-main},  it suffices to prove that the $Y_{a}$'s of Theorem~\ref{thm-main} satisfy both conditions (\ref{item:thm:Gromov-systolic:first}) and (\ref{item:thm:Gromov-systolic:second}) of Theorem~\ref{thm-Gromov-criteria-for-top-exp}. To this end we will prove the following isoperimetric result(s):

\begin{thm}\label{thm-isoperimetric -inequalities}
Fix $q \gg 0$. Let $F$ be a local field of residue degree $q$, $\calB=\tilde{A}_3(F)$ the $3$-dimensional Bruhat-Tits building associated with $PGL_4(F)$. Then there exist $\eta_0,\eta_1,\eta_2, \epsilon_0, \epsilon_1,\epsilon_2$ all greater than $0$ such that: if $X$ is a non-partite Ramanujan quotient of $\calB=\tilde{A}_3(F)$ and $\alpha \in C^i(X,\F_2)$, $0 \leq i \leq 2$, a locally minimal cochain with $||\alpha|| \leq \eta_i$ then $|| \delta_i(\alpha)|| \geq \epsilon_i ||\alpha||$.
\end{thm}
%\tnote{update dependence on $q$}

The concept of locally minimal cochain is quite central in our work, but too technical to be defined in the introduction - see Definition~\ref{def-loc-minimal} below. We believe that the above six constants can be made to be independent of $q$, but as of now we know this only for some of them.
%We just mention here that for every $f \in C^i$ there exists $\alpha \in C^{i-1}$ with $||\alpha|| \leq c_0 ||f||$ such that $f' = f + \delta \alpha$ is locally minimal, where $c_0$ is some constant depending only on $q$.

Theorem~\ref{thm-isoperimetric -inequalities} is best possible: it is not true without the assumption that $||\alpha||$ is small. As mentioned before, the Ramanujan complexes are in general {\em not} coboundary expanders: For $i=1$ or $i=2$ (but not for $i=0$), it is possible to find a locally minimal $\alpha \in C^i \setminus B^i$ with $\delta_i(\alpha) = 0$.
%Note that if we would prove Theorem~\ref{thm-isoperimetric -inequalities} without the restriction of $||f||$ to be small then it would prove that the Ramanujan complexes are coboundary expanders. But, as said before, this is not true!
%For $i=1$ or $i=2$ (but not for $i=0$), it is even possible to find $f \in C^i \setminus B^i$ with $\delta_i(f) = 0$.

It is interesting to observe that in order to prove that the $Y_{a}$'s of Theorem~\ref{thm-main} are topological expanders we have to prove the above isoperimetric results for the $X_{a}$'s. Of course at level $i=0,1$, $Y_{a}$ and $X_{a}$ are the same, but for $i=2$, $\delta_2$ is zero on $C^2(Y_{a},\F_2)$ but non-zero on $C^2(X_{a}, \F_2)$. We refer the reader to Section~\ref{section-main-proof} to see how the information on $X_{a}$ helps to deduce the desired result for $Y_{a}$.

Our main technical result is Theorem~\ref{thm-isoperimetric -inequalities}. Before proving this theorem, we will give a "baby version" of it for $2$-dimensional Ramanujan complexes. This case is easier (for reasons to be understood in Section~\ref{section-one-cochain-three-dim-proof} and Section~\ref{section-three-dim-coboundary-exp}) though still far from trivial, and the main ideas of the proof of Theorem~\ref{thm-isoperimetric -inequalities} show up already there. It also has some independent interest (see Corollary~\ref{cor-systolic-two-dim} below and the discussion following it). In this case we can also give a very sharp estimates on the constants, which are also independent of $q$:

%Theorem~\ref{thm-isoperimetric -inequalities} suffices in order to prove that $Y$, the $2$-skeleton of $X$, has the topological overlapping property. It implies that $\mu_i$'s of Definition~\ref{def-gromov-filling-inverse-norm} are bounded by $\frac{1}{\eta_i} + c_0$. Indeed, let $\beta \in B^i$. As above, replace $\beta$ by $\beta'=\beta + \delta \alpha$ where $\alpha \in C^{i-1}$ with $||\alpha|| \leq c_0 ||\beta||$ and $\beta'$ is locally minimal. If $||\beta'|| \leq \eta_i$ then $|| \delta (\beta') || > \epsilon_i || \beta'||$. But, $\delta (\beta') =\delta(\beta) + \delta \delta (\alpha) = 0$ and hence $||\beta'|| = 0 $ i.e., $\beta = \delta \alpha$ and $||\alpha|| \leq c_0 ||\beta||$ so we are done. If $||\beta'|| > \eta_i$ then as $\beta'$ is also a coboundary, $\beta' = \delta \alpha'$ with $|| \alpha'|| \leq 1$ (by definition), and so $\beta = \delta(\alpha' + \alpha)$ and $||\alpha' + \alpha|| \leq  \frac{1}{\eta_i}||\beta'||  + c_0 ||\beta|| \leq (\frac{1}{\eta_i}  + c_0 ) ||\beta||$ and again we are done.
%
%
%The following theorem provides another, in fact, much stronger, isoperimetric  result:

\begin{thm}
\label{thm-two-dim-coboundary-exp}
Given $\epsilon_0> 0$, there exists $q(\epsilon_0) \in \N$ and $0 < \epsilon  \in\R$ such that: Let $X$ be a $2$ dimensional Ramanujan complex, a quotient of the Bruhat-Tits building of $PGL_3(\F_q((t)))$ with $q \geq q(\epsilon_0)$. Let $\alpha \in C^1(X,\F_2)$ be a locally minimal $1$-cochain, with $|| \alpha || < \frac{1}{4(1+\epsilon_0)}$. Then $||\delta_1(\alpha)|| \geq \epsilon ||\alpha||$.

%Let $X$ be a $2$ dimensional Ramanujan complex on $n$ vertices whose $1$-skeleton is a $Q$-regular graph, with $Q=2(q^2+q+1)$ and $\alpha \in C^1(X,\F_2)$, a locally minimal $1$-cochain.
%Then, for every $\epsilon_0 > 0$, there exists $\epsilon_2 > 0$ such that if $|| \alpha || < \frac{1}{4(1+\epsilon_0)}$ then
%$||\delta_1(\alpha)|| \geq \epsilon_2 ||\alpha||$, provided $q \gg 0$ (i.e., $q \geq q(\epsilon_0)$).
%Here, $\delta_1$ is the coboundary map $\delta_1: C^1(X,\F_2) \rightarrow C^2(X,\F_2)$.
\end{thm}

Every locally minimal $\alpha \in C^1(X,\F_2)$ satisfies $||\alpha|| \leq \frac{1}{2}$ (see Section~\ref{subsection-notions-of-minimality}). So, the theorem says that if $\alpha$ has slightly less than half of the maximal number of edges of a locally minimal cochain, its coboundary is "large". This is essentially best possible, as we will show (Proposition~\ref{prop-non-trivial-first-cohomology}) that there are non-zero locally minimal cochains $\alpha$ with $\delta(\alpha) = 0$.

%Theorem~\ref{thm-two-dim-coboundary-exp} can be extended to $d$-dimensional Ramanujan complexes with less impressive (but still explicit, in principle) constants.

%But, more interesting is the extension to the case of $\alpha$ being $2$-cochain.
%
%\begin{thm}\label{thm-three-dim-coboundary-exp}
%Let $X$ be a $3$-dimensional Ramanujan complex and $\alpha \in C^2(X,\F_2)$ a locally minimal $2$-cochain. Then, there exists $\epsilon_2 > 0$ s.t. if $||\alpha|| \leq q^3 n$, then $||\delta_2(\alpha)|| \geq \epsilon_2 q ||\alpha||$ provided $q \gg 0$.
%Here $\delta_2: C^2(X,\F_2) \rightarrow C^3(X,\F_2)$ is the coboundary map.
%\end{thm}

The above isoperimetric  results and their proofs give various (mod $2$) systolic inequalities. These seem to be the first {\em linear} lower bound on cohomological systole. Such lower bounds are of importance for quantum error correcting codes. They are needed for the estimate of the distance of the so called CSS-quantum codes~\cite{MeyerFreedmanLuo,Zemor,LubotzkyGuth}. For example we have:

\begin{cor}
\label{cor-systolic-two-dim}
Given $\epsilon_0 > 0$, there exists $q(\epsilon_0) \in \N$ such that: If $X$ is a non-partite Ramanujan complex of dimension $2$, a quotient of the Bruhat-Tits building of $PGL_3(\F_q((t)))$ with $q \geq q(\epsilon_0)$, and $\alpha \in Z^1(X,\F_2) \setminus B^1(X,\F_2)$, i.e., a $1$-cocycle which represents a non-trivial cohomology class, then $||\alpha|| \geq \frac{1}{4(1+\epsilon_0)}$.
\end{cor}

Ramanujan complexes of dimension $2$ are in many ways non-archimedean analogue of $2$-dimensional manifolds, i.e. Riemann surfaces. It is interesting to compare the systolic behavior. For arithmetic hyperbolic surfaces, the $1$-homological systole is logarithmic and by Poincare duality the same holds for the $1$-cohomological systole. But, for the Ramanujan complexes the $1$-homological systole is logarithmic and the $1$-chomological systole is linear!
%It is interesting to compare this situation with systolic invariants of $2$-dimensional manifolds.
%For arithmetic hyperbolic Riemann surfaces one can show that their $1$-homological systole is logarithmic in the volume, and hence the same holds for the $1$-cohomological systole by the Poincare-duality. It is interesting to see that for $2$-dimensional Ramanujan complexes (which are in many ways the analogue of arithmetic hyperbolic surfaces when $\R$ is replaced by a non archimedean local field), the $1$-homological systole is logarithmic in the volume but the $1$-cohomological systole is linear (by Corollary~\ref{cor-systolic-two-dim}).

%For these Poinca\'{r}e duality implies that their systolic invariants for the $1$-homology and the $1$-cohomology are the same (as $1=2-1$). For Ramanujan complexes, one can show (see Section~\ref{section-buildings} below) that (sometimes) there are non-trivial $1$-homology classes with logarithmic size support, while Corollary~\ref{cor-systolic-two-dim} says that all non trivial $1$-cohomology class are of linear size. So $X$ behaves dramatically different than manifolds.

See more in Section~\ref{section-three-dim-coboundary-exp} where such linear lower bounds are proved also for $2$-cocycles of $3$-dimensional complexes. %(See Corollary~\ref{cor-systolic-three-dim}).

The paper is organized as follows. In Section~\ref{section-homological-defs} we introduce the basic cohomological notations (over $\F_2$), and the (local)-minimality of cochains. In Section~\ref{section-buildings}, we review the spherical and affine buildings and the properties of Ramanujan complexes. In Section~\ref{section-main-proof} we show how, assuming Theorem~\ref{thm-isoperimetric -inequalities}, one can prove Theorem~\ref{thm-main}, leaving the (quite technical) proof of Theorem~\ref{thm-isoperimetric -inequalities} to Sections ~\ref{section-one-cochain-three-dim-proof} (cases $i=0,1$) and~\ref{section-three-dim-coboundary-exp} (case $i=2$). To illustrate first the main ideas of the proof of Theorem~\ref{thm-isoperimetric -inequalities} in an easier case, we give in Section~\ref{section-two-dim-proof} a proof of Theorem~\ref{thm-two-dim-coboundary-exp}. In Section~\ref{section-Serre-conj-implications} we show that our results combined with Serre's conjecture~\cite{Serre} on the congruence subgroup property give bounded degree $2$-dimensional coboundary expanders. Serre's conjecture has been proven in most cases, but unfortunately, not in the cases we need here. In fact, what we need is only the vanishing of $H^1(\Gamma, \F_2)$ for suitable congruence subgroups, which is a corollary of Serre's conjecture, and possibly easier than it. See Section~\ref{section-Serre-conj-implications} for more.

\section{Coboundary expansion and overlapping}
\label{section-homological-defs}

In this section we review some notations and results on general simplicial complexes.

\subsection{Expansion of graphs}
Let $X=(V,E)$ be a finite connected graph, $A=A_X$ its adjacency matrix and $\Delta$ its Laplacian, i.e., $\Delta:L^2(X) \rightarrow L^2(X)$ is defined by $\Delta(f)(v)=deg(v)f(v)-\sum_{y \thicksim v}f(y)$ where the sum is over the neighbors of $v$. If $X$ is $k$-regular then $\Delta=kI-A$.
It is well known that the eigenvalues of $\Delta$ (and $A$) are intimately connected with expansion properties of $X$. Let us mention a variant which we need, due to Alon and Milman~\cite[Prop 4.2.5]{LubotzkyBook}.

\begin{proposition}~\label{prop-cheeger} Let $\lambda_1(X)$ be the smallest positive eigenvalue of $\Delta$. Then, for every subset $W \subseteq V$,
\begin{enumerate}
\item $|E(W,\bar{W})| \geq \frac{|W||\bar{W}|}{|V|} \lambda_1(X)$,

where $E(W,\bar{W})$ denotes the set of edges from $W$ to its complement $\bar{W}$.

\item The Cheeger constant $h(X)$ satisfies:

$h(X):=\mbox{min}_{W \subseteq V} \frac{|E(W,\bar{W})|}{\mbox{min}(|W|,|\bar{W}|)} \geq \frac{\lambda_1(X)}{2}$.

\item If $X$ is $k$-regular then

$ E(W):=E(W,W)=\frac{1}{2}(k|W|-E(W,\bar{W})) \leq \frac{1}{2}(k-\frac{|\bar{W}|}{|V|}\lambda_1(X))|W|$.
\end{enumerate}
\end{proposition}

We will also need the following variant for bipartite bi-regular graphs, whose proof can be found for example in~\cite{EvraGLubotzky}.

\begin{proposition}(Mixing Lemma for bipartite bi-regular graphs)~\label{prop-mixing-for-bipartite}
Let $X=(V',V'',E)$ be a bipartite $(k',k'')$-bi-regular finite graph. Then, for every subsets $A \subseteq V'$, $B \subseteq V''$,
$$|E(A,B)| - \frac{\sqrt{k'k''}|A||B|}{\sqrt{|V'||V''|}}| \leq \lambda(X)\sqrt{|A||B|},$$ where $\lambda(X)$ is the second largest eigenvalue of the adjacency matrix of $X$.
\end{proposition}

\subsection{Coboundary expansion of simplicial complexes}\label{subsection:coboundary-exp-simp-complexes}
Let us now pass to the higher dimensional case, so from now on $X$ will be a finite $d$-dimensional simplicial complex with a set of vertices $V=X(0)$. Namely, $X$ is a set of subsets of $V$ with $F_1 \in X$ and $F_2 \subseteq F_1$ implies $F_2 \in X$ and $\mbox{max}\{|F| \mbox{ } | F \in X\} = d+1$. For $F \in X$, $\mbox{dim}(F) := |F| - 1$, and $X(i)$ denotes the set of cells of dimension $i$, i.e., those $F$ with $|F| = i+1$. So, $X(-1)=\{\emptyset\}$. By $X^{(i)}$ we denote the $i$-skeleton of $X$, i.e., $X^{(i)}=\cup_{j \leq i}X(j)$. We say that $X$ is a {\em pure} complex if all maximal cells (facets) in $X$ are of the same dimension. All the simplicial complexes dealt in this paper are pure.
Let $C^i=C^i(X,\F_2)$ be the space of $i$-cochains, i.e., $\{f:X(i) \rightarrow \F_2\}$. It will be sometimes convenient to think of $\alpha \in C^i$ as a collection of $i$-cells and we will denote its cardinality by $|\alpha|$.

For $\sigma \in X(i)$, we denote

\begin{eqnarray}\label{eqarray-c(sigma)}
c(\sigma) & := & |\{\tau \in X(d) | \sigma \subseteq \tau \}|
\end{eqnarray}

and

\begin{eqnarray}\label{eqarray-w(sigma)}
w(\sigma) & :=  & \frac{c(\sigma)}{{d+1 \choose i+1} |X(d)|}.
\end{eqnarray}

Note that $\sum_{\sigma \in X(i)}w(\sigma) =1$. The weight function $w$ on $X(i)$ defines a norm on $C^i(X,\F_2)$ by

\begin{eqnarray}\label{eqarray-||alpha||}
||\alpha|| & := & \sum_{\sigma \in \alpha} w(\sigma),
\end{eqnarray}

where again we consider $\alpha$ as a collection of $i$-cells. Note that $||\alpha|| \leq 1$.

Let $\delta = \delta_i:C^i \rightarrow C^{i+1}$ be the coboundary map, i.e., for $\sigma \in X(i+1)$
$$\delta_i(\alpha)(\sigma)=\sum_{\tau \subseteq \sigma, \mbox{ dim}\tau = i}\alpha(\tau) .$$

Note that as we are working over $\F_2$, we can ignore the issue of orientation. It is easy to see that $\delta_{i+1} \circ \delta_i =0$. Denote $B^i=B^i(X,\F_2)$ and $Z^i=Z^i(X,\F_2)$ the $i$-coboundaries ($\mbox{Im} \delta_{i-1}$) and the $i$-cocycles ($\mbox{Ker} \delta_i$), respectively. Then, $B^i \subseteq Z^i$ and $H^i=H^i(X,\F_2)=Z^i/B^i$ is the (reduced) $i$-cohomology group over $\F_2$.

%\begin{definition}(Coboundary expansion)~\label{def-coboundary-expansion} Let $X$ be a finite $d$-dimensional simplicial complex. For $i=0, \cdots, d-1$, denote
%$$\epsilon_i = \epsilon_i(X) =\mbox{min}_{\alpha \in C^i \setminus B^i}\frac{||\delta_i (\alpha)||}{||[\alpha] ||} , $$ which is called the {\em $i$-coboundary expansion} of $X$. Here, $[\alpha]$ is the coset $\alpha+B^i$ and $||[\alpha]||= \mbox{min}\{||\gamma|| \mbox{ } | \gamma \in [\alpha] \}$.
%
%%The {\em normalized $i$-coboundary expansion} of $X$ is
%%$$\tilde{\epsilon}_i = \frac{|X(i)|}{|X(i+1)|} \epsilon_i(X)=\mbox{min}_{\alpha \in C^i \setminus B^i}\frac{||\delta_i (\alpha)||}{||[\alpha]||}.$$
%\end{definition}
%
%As $B^0=\{0,\mathbf{1}\}$, one can check that $\epsilon_0(X)$ is equal to the normalized Cheeger constant of the $1$-skeleton $X^{(1)}$ of $X$. Thus, bounds on $\epsilon_i$ serve as a possible definition for high dimensional expansion as was first suggested by Linial amd Meshulam~\cite{LinialMeshulam} (see~\cite{LubotzkyJapan} for a survey).
%
%A closely related notion to the coboundary expansion has been studied by Gromov.
%
%\begin{definition}(The inverse filling norm)~\label{def-inverse-filling-norm}
%The {\em $i$-inverse filling norm} of $X$ for $i=0, \cdots, d-1$ is defined as
%$$\mu_i =\mbox{max}_{\beta \in B^{i+1}}\{ \frac{1}{||\beta||} \mbox{min}_{\alpha \in C^i , \delta_i(\alpha)=\beta}\{||\alpha||\}\}.$$
%\end{definition}
%
%One can check that if $H^i=0$ then $\mu_i = \frac{1}{\tilde{\epsilon}_i}$, but it is still possible that $\epsilon_i =0$ while $\mu_i < \infty$.

\subsection{Notions of minimality}\label{subsection-notions-of-minimality}
Given a pure simplicial complex $X$ of dimension $d$ and $\tau$ a simplex of $X$, the link $X_{\tau}$ of $X$ at $\tau$ is the set of all sets of the form $\sigma \setminus \tau$, where $\sigma \in X$ and $\tau \subseteq \sigma$. Then, $X_{\tau}$ is a complex, with set of vertices $X(0)\setminus \tau$, of dimension $\mbox{dim}(X)-\mbox{dim}(\tau)-1$. In particular, for a vertex $v$, the link $X_v$ of $v$ is of dimension $d-1$. A cochain $\alpha \in C^i(X,\F_2)$ defines a cochain $\alpha_{v} \in C^{i-1}(X_v,\F_2)$, by $\alpha_v(\sigma \setminus \{v\})=\alpha(\sigma)$ for every $\sigma \in X(i)$ containing $v$.

Throughout this paper we assume that our simplicial complexes are homogenous in the following sense: the structure of $X_v$ is independent of $v$, i.e., the links of all the vertices are isomorphic. In particular $|X_v(d-1)|$ is independent of $v$. Under this assumption we have:

\begin{lemma}
\label{lemma-loc-min}
For $\alpha \in C^i(X,\F_2)$, $||\alpha||=\frac{1}{|X(0)|}\sum_{v \in X_0}||\alpha_v||$.
\end{lemma}

\begin{proof}
By our assumption $|X(d)| = \frac{1}{d+1}|X(0)||X_v(d-1)|$.
Now
\begin{eqnarray}
\sum_{v \in X(0)}||\alpha_v|| & = & \sum_{v \in X(0)}\sum_{\sigma \in \alpha_v} w(\sigma) = \sum_{v \in X(0)}\sum_{\sigma \in \alpha_v} \frac{c_{X_v}(\sigma)}{{d \choose i} |X_v(d-1)|} \\
                             & =  & \sum_{\sigma \in \alpha}\sum_{v \in \sigma} \frac{c_{X}(\sigma)}{{d \choose i} \frac{d+1}{|X(0)|} |X(d)|} = \sum_{\sigma \in \alpha}\frac{(i+1)c_X(\sigma)|X(0)|}{{d \choose i}(d+1)|X(d)|} \\
                             &  = & |X(0)|\sum_{\sigma \in \alpha}\frac{c_X(\sigma)}{{d+1 \choose i+1}|X(d)|} = |X(0)| \cdot ||\alpha||.
\end{eqnarray}
\end{proof}

Let us now discuss few notions of minimality.

\begin{definition}\label{def-loc-minimal}.

\begin{enumerate}

\item A cochain $\alpha \in C^i(X,\F_2)$ is called {\em minimal} if it is of minimal norm in its class modulo $B^i(X,\F_2)$, i.e., $||\alpha || \leq ||\alpha + \delta_{i-1}\gamma ||$ for every $\gamma \in C^{i-1}$. This is equivalent to $||\alpha|| = \mbox{dist}(\alpha, B^i)$ where the distance between a vector $\alpha$ and a subspace $W$ is defined as $\mbox{dist}(\alpha,W) = \mbox{min}\{||\gamma|| \mbox{ } | \mbox{ } \alpha + \gamma \in W\}$.
%A cochain $\alpha \in C^i$ is locally minimal if for every $(i-1)$-cell $\tau$, $||\alpha_{\tau}||\leq \frac{1}{2}|| X_{\tau}(0)||$, where
%$X_{\tau}(0) = \{\sigma \in X | \mbox{dim } \sigma  = i, \tau \subseteq \sigma \}$ and
%$\alpha_{\tau}=\{ \sigma \in X_{\tau}(0) | \alpha(\sigma) \neq 0 \}$.

%Given $\alpha \in C^i$, we can replace it by a locally minimal representative of the same coset mod $B^i$ by the following process.
%If $\alpha$ is not locally minimal at $\tau \in X(i-1)$, replace $\alpha$ by $\alpha' = \alpha + \delta_{i+1}1_{\tau}$. Clearly, $||\alpha'|| < ||\alpha||$, so the process terminates after finitely many steps (in fact, at most $||\alpha||$ steps) with a locally minimal representative.
%It is possible that along the way some vertices in which $\alpha$ was locally minimal are temporarily not locally minimal but eventually, the process converges to a locally minimal cocycle in $\alpha + B^i$.

\item A cochain $\alpha \in C^0(X,\F_2)$ will be called {\em locally minimal} if it is minimal while for $i \geq 1$, $\alpha\in C^i(X,\F_2)$ is called {\em locally minimal} if for every vertex $v$ of $X$, $\alpha_v$ is a minimal $(i-1)$-cochain in $C^{i-1}(X_v,\F_2)$.
\end{enumerate}
\end{definition}

Every minimal cochain is locally minimal, but not  every locally minimal cochain is minimal\cite{BenAriVishne}. To prove coboundary expansion, one can, in principle, consider only minimal cochains. But, in our work, it is crucial that Theorem~\ref{thm-isoperimetric -inequalities} is proved for the more general case of locally minimal cochain. This is used in an essential way in Section~\ref{section-main-proof} to deduce topological expansion for the $Y_a$'s of Theorem~\ref{thm-main} from the isoperimetric inequalities proved for the $X_a$'s. Every $\alpha \in C^i$ is equivalent modulo $B^i$ to a locally minimal one, in fact, even one which is not too far from it.

\begin{proposition}\label{prop-loc-min-properties}
Assume $X$ is a finite homogeneous pure $d$-dimensional complex. In particular, every $v \in X(0)$ lies in $m(i)$ $i$-simplicies. Then:
\begin{enumerate}
\item\label{item-one-prop-loc-min-properties} For every $\alpha \in C^i(X,\F_2)$, there exists a locally minimal $\alpha' \in C^i(X,\F_2)$ with $\alpha' \equiv \alpha \mbox{mod }B^i(X,\F_2))$, $||\alpha'|| \leq  ||\alpha||$ and $\alpha' = \alpha +\delta_{i-1}(\gamma)$ where $\gamma \in C^{i-1}(X,\F_2)$ with $||\gamma|| \leq c ||\alpha||$, where $c=\frac{d+1-i}{i+1}m(i-1)$.
\item\label{item-two-prop-loc-min-properties} If for some $i \leq d$, $c(\sigma)$ (see Equation (\ref{eqarray-c(sigma)}) in Section~\ref{subsection:coboundary-exp-simp-complexes}) is constant on the simplicies $\sigma \in X(i)$, then for $\alpha \in C^i(X,\F_2)$, $||\alpha||$ is the normalized counting norm, i.e., $||\alpha|| = \frac{|\alpha|}{|X(i)|}$. If $\alpha$ is also locally minimal, then for every vertex $v \in X(0)$, if we consider $\alpha_v$ as a set of (i-1)- cells in $X_v(i-1)$, we have $|\alpha_v| \leq \frac{|X_v(i-1)|}{2}$.
\end{enumerate}
\end{proposition}
Note, if $X$ is homogenous them $c(v)$ is constant on vertices but not necessarily so for $i$-cells.

\begin{proof} If $\alpha$ is locally minimal there is nothing to prove. If not, then for some $v$, $||\alpha_v +\gamma|| < ||\alpha_v||$ for some $\gamma\in B^{i-1}(X_v,\F_2)$. Define $\tilde{\gamma} \in C^i(X,\F_2)$ by $\tilde{\gamma}(\sigma)=0$ if $v \notin \sigma$ and $\tilde{\gamma}(\sigma) = \gamma(\sigma \setminus \{v\})$ if $v \in \sigma$, where $\sigma \in X(i)$. One can easily check that $\tilde{\gamma}_v =\gamma$. As $\gamma\in B^{i-1}(X_v,\F_2)$, we have $\tilde{\gamma} \in B^i(X,\F_2)$, in fact, if $\gamma=\delta_{i-2}(\eta)$ for some $\eta \in X_v(i-2)$, then $\tilde{\gamma}=\delta_{i-1}(\tilde{\eta})$. Now, replace $\alpha$ by $\alpha+\tilde{\gamma}$. By doing so, $\alpha+\tilde{\gamma} \equiv \alpha (\mbox{mod }B^i(X,\F_2))$. Moreover, $||\alpha+\tilde{\gamma}|| < ||\alpha||$. It is clear that $||(\alpha+\tilde{\gamma})_v|| < ||\alpha_v||$, but some care is needed (and we can not just apply Lemma~\ref{lemma-loc-min}) as $\tilde{\gamma}$ also influences other vertices. But, adding $\tilde{\gamma}$ changes the value of $\alpha$ only on simplicies which contain the vertex $v$, and on them it decreases their contribution to the norm of $\alpha$, i.e.,
$$\sum_{v \in \sigma \in \alpha+\tilde{\gamma}}w(\sigma) < \sum_{v \in \sigma \in \alpha}w(\sigma),$$ and hence $||\alpha+\tilde{\gamma}|| < ||\alpha||$.

The above process terminates since $||\alpha||$ can get only finitely many values, so eventually we replace $\alpha$ by a locally minimal cochain in the same class modulo $B^i(X,\F_2)$. In fact, the process terminates after at most ${d+1 \choose i+1}|X(d)| \cdot ||\alpha||$ steps, since for every $i$-cochain in $C^i(X,\F_2)$, its norm is an integral multiply of $\frac{1}{{d+1 \choose i+1}|X(d)|}$.

In each step the local change is by $\tilde{\gamma} = \delta(\tilde{\eta})$, and the norm of $\tilde{\eta}$ is at most $\frac{m(i-1)}{{d+1 \choose i}|X(d)|}$. The number of steps is at most ${d+1 \choose i+1}|X(d)| \cdot ||\alpha||$ and so the total change $\gamma$ is of norm at most $\frac{d+1-i}{i+1}m(i-1) ||\alpha||$.

\bigskip

For the proof of the second part, note first that $w(\sigma)$ is constant on $X(i)$ and $\sum_{\sigma \in X(i)}w(\sigma)=1$ and hence the norm on $C^i(X,\F_2)$ is simply the normalized counting norm. Now, we also have that all the ($i-1$)-cells of $X_v$ have the same weight in $X_v$. If $\alpha \in C^i(X,\F_2)$ is locally minimal and for some $v \in X(0)$, $\alpha_v$ contains more than half of the $(i-1)$-cells of $X_v$, then for some $\tau \in X_v(i-2)$, $\alpha_v$ contains more than half of the $(i-1)$-cells in $X_v$ containing $\tau$. This implies that $||\alpha_v + \delta_{i-2}(\tau)|| < ||\alpha_v||$ in contradiction to the minimality of $\alpha_v$, i.e., in contradiction to the local minimality of $\alpha$.
\end{proof}

\section{Buildings and Ramanujan complexes}
\label{section-buildings}

In this section we review some notations and results on buildings and their quotients and prepare some technical results to be used later. We start with spherical buildings.

\subsection{Spherical buildings}\label{subsection-spherical-buildings}
Let $W=\F_q^r$ be an $r$-dimensional vector space over the finite field of order $q$. Denote by $S(r,q)$ the spherical building associated with $PGL_r(\F_q)$, i.e., the flag complex of $\F_q^r$. This is the simplicial complex whose vertices are all the non-zero proper subspaces of $W$, and $i+1$ such subspaces $u_0, \cdots, u_i$ form an $i$-cell if $u_0 < u_1< \cdots < u_i$. This is a finite simplicial complex of dimension $r-2$, which is known to be homotopic to a bouquet of $(r-2)$-dimensional spheres. In particular, $H^i(S(r,q),\F_2)=0$ for every $i=1,\cdots, r-3$. It was shown by Gromov that these complexes have the overlapping properties~\cite[p. 457]{Gromov}, showing along the way that they are coboundary expanders.
%In fact, he showed:
%
%\begin{thm}~\cite[p. 457]{Gromov} Given $r$ there is a finite number $K(r)$ such that all the inverse filling norms $\mu_i(S(r,q))$ for $i=1,\cdots, r-3$ (see Definition~\ref{def-inverse-filling-norm}) are all bounded by $K(r)$.
%\end{thm}
%
%As $H_i=0$, $\mu_i=\frac{1}{\tilde{\epsilon}}$, so the theorem gives a universal lower bound on $\tilde{\epsilon}_i$. Spelling it out for $\epsilon_i$ (without the normalization), we get:
%
%\begin{cor}~\label{cor-Gromov-spherical-result} Given $r$, there is $\epsilon(r)>0$, such that for every $i=0, \cdots, r-3$, $\epsilon_i(S(r,q)) \geq \epsilon(r)q$.
%\end{cor}
%
%\begin{proof} $\tilde{\epsilon}_i(X) = \frac{|X(i)|}{|X(i+1)|}\epsilon_i(X)$ and for $X=S(r,q)$, $\frac{|X(i)|}{|X(i+1)|} \geq q+1$.
%\end{proof}
%
%We will need another fact on $S(r,q)$.

\begin{thm}\label{prop-S(4,q)} There exists some constant $\epsilon(r) > 0$ such that $\epsilon_i(S(r,q)) \geq \epsilon(r)$ for every $i=0,\cdots,r-1$ and every prime power $q$.
\end{thm}

For a proof of Theorem~\ref{prop-S(4,q)} see~\cite{LubotzkyMeshulamMozes}. We will need only the case $r=4$ and $i=1$. But, we will need few more facts on $S(r,q)$ for small values of $r$.
Let $S(r,q)^{(1)}$ be the $1$-skeleton of $S(r,q)$, i.e., the graph whose vertices are the non-zero proper subspaces of $\F_q^r$ with two such subspaces are incidence iff one is contained in the other. For $r=3$ this is the well studied "points versus lines graph" of the projective plane. This is a $(q+1)$-regular graph whose eigenvalues (of the adjacency matrix) are $\pm(q+1)$ and $\pm\sqrt{q}$, the later with high multiplicity. In particular, these are Ramanujan graphs (of unbounded degree). For $r > 3$, $S(r,q)^{(1)}$ is not regular anymore. Let us look closely at the case $r=4$.

%E.g., for $r=4$, there are subspaces of dimensions $1$, $2$ and $3$. Those of dimension $1$ and $3$ give $2(\frac{q^4-1}{q-1})\sim 2q^3$ vertices of $S(r,q)^{(1)}$, each of degree $2(\frac{q^3-1}{q-1})\sim 2q^2$, while those of dimension $2$ give rise to $[{4 \choose 2}]_q = (q^2+1)(q^2+q+1)\sim q^4$ vertices, each of degree $2(q+1)$.

Let $Z=S(4,q)^{(1)}$ be the $1$-skeleton of the spherical building of $\F_q^4$, i.e., $Z$ is the graph whose set of vertices is $M=M_1 \cup M_2 \cup M_3$ where $M_i$ is the set of subspaces of $\F_q^4$ of dimension $i$. Note that $|M_1| = |M_3|=\frac{q^4-1}{q-1} = q^3+q^2+q+1$ while $|M_2|=\frac{\frac{q^4-1}{q-1} \cdot \frac{q^3-1}{q-1} }{q+1}\sim q^4$. Two vertices are connected by an edge if one subspace is contained in the other. One easily checks that
every vertex in $M_1$ (resp. $M_3$) is connected with $q^2+q+1$ vertices in $M_2$ and with $q^2+q+1$ vertices in $M_3$ (resp. $M_1$), so its degree is $2(q^2+q+1)$. On the other hand, the degree of a vertex in $M_2$ is $2(q+1)$, half of them go to $M_1$ and half to $M_3$.

The following technical lemma will be needed in Section~\ref{section-three-dim-coboundary-exp}.
\begin{lemma}~\label{lemma-neighborhood-of-thin-vertex-defs} Let $T=T_1 \cup T_2 \cup T_3 \subseteq M$ be a subset of the vertices of $Z$ with $T_i \subseteq M_i$. Assume that with every $t \in T$, a set of edges $E(t)$, coming from $t$, is given and let $\tilde{E}=\bigcup_{t \in T}E(t)$. Assume also
\begin{itemize}
\item $|T_1|, |T_3| \leq q^{2.75}$ and $|T_2| \leq q^{3.7}$.
\item for $t \in T_1 \cup T_3$, $|E(t)| > q^{1.8}$ and for $t \in T_2$, $|E(t)| > q^{0.9}$.
\end{itemize}

Then,
%Denote by $\tilde{E}^{out} = \{ e \in \tilde{E} | \mbox{one endpoint of e is outside T } \}$ then
\begin{eqnarray}~\label{eqnarray-lemma-neighborhood-of-thin-vertex-defs}
\frac{|E(T,T)|}{|\tilde{E}|} = o_q(1) .
\end{eqnarray}
I.e, there exists $\epsilon(q)$ with $\epsilon(q) \rightarrow 0$ when $q \rightarrow \infty$, s.t. for every choice of $T$ and $\{E(t) | t \in T \}$ as above, $\frac{|E(T,T)|}{|\tilde{E}|} \leq \epsilon(q)$.
\end{lemma}

\begin{proof}
%We will prove that when $q \rightarrow \infty$,
%\begin{eqnarray}\label{eqnarray-spherical-building-eigenvalues-calc}
%\frac{|E(T,T)|}{|\tilde{E}|} \rightarrow 0
%\end{eqnarray}
%where $E(T,T)$ is the set of all edges in $Z$ between vertices of $T$. As $|E^{out}| \geq |\tilde{E}| - |E(T,T)|$ the lemma will follow.
As $Z$ is a $3$-partite graph, $E(T,T)=E(T_1,T_2) \cup E(T_2,T_3) \cup E(T_1,T_3)$. It suffices to prove (\ref{eqnarray-lemma-neighborhood-of-thin-vertex-defs}) for each $E(T_i,T_j)$ separately. We can therefore consider the graphs $Z_{i,j}$, $1 \leq i < j \leq 3$ where $Z_{i,j}$ is the bipartite graph whose vertices are $M_i \cup M_j$ and the adjacency relation is as in $Z$. Note that $Z_{1,2}$ and $Z_{2,3}$ are isomorphic and to prove the result for one is like to prove the result for the other. So, we will prove it only for $Z_{1,2}$ and $Z_{1,3}$.

\begin{lemma}~\label{lemma-one-spherical-eigenvalues-calc}
\begin{enumerate}
\item~\label{item-one-lemma-one-spherical-eigenvalues-calc} Let $A$ be the adjacency matrix of the graph $Z_{1,3}$. Then its eigenvalues are $\pm(q^2+q+1)$, each with multiplicity $1$, and $\pm q$ with high multiplicity.
\item~\label{item-two-lemma-one-spherical-eigenvalues-calc} Let $A$ be the adjacency matrix of the graph $Z_{1,2}$. Then, its largest eigenvalue is $\sqrt{(q+1)(q^2+q+1)}$ and the other eigenvalues are either $\pm\sqrt{q^2+q}$ or $0$.
\end{enumerate}
\end{lemma}
\begin{proof}
The matrix $A$ has a block form $A=\left(
                                   \begin{array}{cc}
                                     0 & B \\
                                     B^t & 0 \\
                                   \end{array}
                                 \right)$ and hence $A^2 =\left(
                                   \begin{array}{cc}
                                     BB^t & 0 \\
                                     0 & B^tB \\
                                   \end{array}
                                 \right)$. The eigenvalues of $B^tB$ and $BB^t$ are the same up to multiplicities of zeros. It suffices therefore to analyze $B^tB$. This is the adjacency matrix of the graph $Y$ with vertex set $M_1$ and two subspaces $u$ and $w$ in $M_1$ are connected by $t$ edges if in the original graph there are $t$ paths of length $2$ from $u$ to $w$. Let us now consider separately the two cases.

(\ref{item-one-lemma-one-spherical-eigenvalues-calc}) In $Z_{1,3}$, a subspace $u$ goes to itself in $q^2+q+1$ $2$-paths according to its degree in $Z_{1,3}$. While if $u \neq w$, then $u$ and $w$ are contained in $q+1$ subspaces of dimension $3$. Hence,
$BB^t = (q^2+q+1)I + (q+1)(I-J) = q^2 I + (q+1)J$, where $J$ is the all $1$'s matrix. Now $J$ acts as the zero matrix on $L_0^2(M_1) = \{f:M_1 \rightarrow \R | \sum_{u \in M_1} f(u)=0 \}$ and as $|M_1|I$ on the constant functions. Thus, the eigenvalues of $B^tB$ are $q^2+(q+1)(q^3+q^2+q+1) = (q^2+q+1)^2$ and $q^2$ as claimed, and the same for $BB^t$.

(\ref{item-two-lemma-one-spherical-eigenvalues-calc}) This time $B$ and $B^t$ are not square matrices but the argument is similar. In $Z_{1,2}$ a subspace $u$ in $M_1$ is connected to itself in $q^2+q+1$ $2$-paths. Two different $1$-dimensional subspaces are inside a unique two dimensional subspace and hence $BB^t = (q^2+q+1)I + (J-I) = (q^2+q)I + J$. Arguing as in part one we deduce that the eigenvalues of $BB^t$ are $(q^2+q+1)(q+1)$ and $(q^2+q)$. Thus, the eigenvalues of $A$ are either $\pm \sqrt{(q^2+q+1)(q+1)}$, $\pm\sqrt{q^2+q}$ or $0$.

\end{proof}

%In order to complete the proof of Lemma~\ref{lemma-neighborhood-of-thin-vertex-defs} we will use Proposition~\ref{prop-mixing-for-bipartite} that deals with mixing properties of bipartite regular graphs.
%
%\begin{lemma}(Mixing Lemma for bipartite bi-regular graphs)~\label{lemma-mixing-for-bipartite}
%Let $G=(V',V'',E)$ be a bipartite $(k',k'')$-bi-regular finite graph. Then, for every subsets $A \subseteq V'$, $B \subseteq V''$,
%$$|E(A,B)| - \frac{\sqrt{k'k''|A||B|}}{\sqrt{|V'||V''|}} \leq \lambda(G)\sqrt{|A||B|},$$ where $\lambda(G)$ is the second largest eigenvalue of $G$.
%\end{lemma}

We are ready now to apply Proposition~\ref{prop-mixing-for-bipartite} for the graphs $Z_{1,3}$ and $Z_{1,2}$. Let us start with $G=Z_{1,3}$. I.e., $A=T_1$, $B=T_3$ and by Lemma~\ref{lemma-one-spherical-eigenvalues-calc} (\ref{item-one-lemma-one-spherical-eigenvalues-calc}), $\lambda(G)=q$. Note, $k'=k''=q^2+q+1 \approx q^2$ and $V'=V'' \approx q^3$. By Proposition~\ref{prop-mixing-for-bipartite},
$$E(A,B) \leq \frac{q^2|A||B|}{q^3} + q \sqrt{|A||B|}  = \frac{|A||B|}{q} + q \sqrt{|A||B|}.$$
On the other hand, up to a factor of $2$, we have
$$|\tilde{E}|\simeq \sum_{t \in T_1 \cup T_3}|E(t)| \geq q^{1.8}|A| + q^{1.8}|B|.$$
Let us separate into two cases: $|A| < |B|$ and $|A| \geq |B|$.
In the first case,
$$ \frac{|E(A,B)|}{|\tilde{E}|} \leq \frac{\frac{|B^2|}{q} + q \sqrt{|B||B|}}{q^{1.8}|B|} =\frac{|B|}{q^{2.8}} + \frac{1}{q^{0.8}}.$$
As $|B|=|T_3|$ was assumed to be less than $q^{2.75}$, the ratio goes to $0$ with $q \rightarrow \infty$ as needed.
The second case, i.e., $|A| \geq |B|$ is symmetric.

Let us now consider the second graph $G=Z_{1.2}$, $A=T_1$, $B=T_2$, $k'=q^2+q+1$, $k''=q+1$, $V' \approx q^3$, $V'' \approx q^4$ and by Lemma~\ref{lemma-one-spherical-eigenvalues-calc}(\ref{item-two-lemma-one-spherical-eigenvalues-calc}), $\lambda(G) \leq 2q$. Thus,
$$E(A,B) \leq \frac{\sqrt{q^3}|A||B|}{\sqrt{q^7}} + 2q \sqrt{|A||B|}  = \frac{|A||B|}{q^2} + 2q \sqrt{|A||B|}.$$
while
$$|\tilde{E}| \simeq \sum_{t \in T_1 \cup T_2}|E(t)| \geq q^{1.8}|A| + q^{0.9}|B|.$$
Again, we separate the evaluation to two cases:
$q^{1.8}|A| <  q^{0.9}|B|$ and $q^{1.8}|A| \geq  q^{0.9}|B|$. In the first case $|A| < q^{-0.9}|B|$, Thus:
$$ \frac{|E(A,B)|}{|\tilde{E}|} \leq \frac{\frac{q^{-0.9} |B|^2 }{q^2} + 2q \sqrt{q^{-0.9}|B||B|}}{q^{0.9}|B|} =
\frac{|B|}{q^{3.8}} + \frac{2q}{q^{1.35}}.$$
As $|B|< q^{3.7}$, this goes to $0$ when $q \rightarrow \infty$. The second case we consider is when $q^{1.8}|A| \geq  q^{0.9}|B|$, so $|B| \leq q^{0.9}|A|$. Thus,
$$ \frac{|E(A,B)|}{|\tilde{E}|} \leq \frac{\frac{q^{0.9} |A|^2 }{q^2} + 2q \sqrt{q^{0.9}|A||A|}}{q^{1.8}|A|} =
\frac{|A|}{q^{2.9}} + \frac{2q^{1.45}}{q^{1.8}}.$$
As $|A|< q^{2.75}$, this goes to $0$ when $q \rightarrow \infty$. Lemma~\ref{lemma-neighborhood-of-thin-vertex-defs} is now proven.
\end{proof}

%For the Laplacian of $S(r,q)^{(1)}$, one has the following bound on its eigenvalues.
%
%\begin{proposition}~\label{prop-esimate-lambda-1-in-spherical-buildings} $\lambda_1(S(r,q)^{(1)}) \geq XXX$.
%\end{proposition}

\subsection{Bruhat-Tits buildings and Ramanujan complexes}\label{subsection-Bruhat-Tits buildings}
Let us move now to the Bruhat-Tits buildings. Let $F$ be a non-archimedean local field, i.e., $F$ is either a finite extension of $\Q_p$ or $F=\F_q((t))$, $\calO$ its valuation ring, $m$ the unique maximal ideal in $\calO$, $\pi$ - a generator of $m$ ("uniformaizer"), so $m=\pi \calO$ and $\calO/m=\F_q$.
The Bruhat-Tits building $B=\tilde{A}_d(F)$ is an infinite simplicial complex defined as follows. An $\calO$-lattice $L$ of $V=F^{d+1}$ is a finitely generated $\calO$-submodule of $V$ which spans $V$. Two such lattices $L_1$ and $L_2$ are equivalent if there exists $0 \neq t \in F$ such that $t L_1= L_2$. The vertices of $\tilde{A}_d(F)$ are the equivalence classes of these lattices and $[L_0],[L_1], \cdots, [L_i]$ form an $i$-cell if there exist representatives $L'_i  \in [L_i]$ s.t. $\pi L'_0 < L'_i < \cdots < L'_2< L'_1 < L'_0$. This is a contractible simplicial complex of dimension $d$, upon which the group $G=PGL_{d+1}(F)$ acts and the action is transitive on the vertices. The $1$-skeleton $B^{(1)}$ of $B$ is a $k$-regular graph where $k$ equals the number of non-zero proper subspaces of $\F_q^{d+1}$ (so for $d=1$, $B$ is the ($q+1$)-regular tree and for general $d$, $k=\sum_{i=1}^{d}{d+1 \choose i}_q\approx q^{\frac{(d+1)^2}{4}}$). In fact, the link of every vertex $v$ of $B$ is isomorphic to the spherical building $S(d+1,q)$, which is a finite simplicial complex of dimension $d-1$. The local properties of $B$ can be read, therefore, from $S(d+1,q)$. For example, every $(d-1)$-cell in $B$ is contained in exactly $(q+1)$ $d$-cells of $B$.

The vertices of $B$ come with a coloring $\tau_B$ in $\Z/(d+1)\Z$, defined as follows. Take an $\calO$-basis $\calB$ for a representative $L'$ of $[L]$ and denote $\tau([L])=\mbox{val}(\mbox{det}\calB)(\mbox{mod}(d+1))$. This is well defined and no adjacent vertices have the same color. This coloring is preserved by the action of $G_0 = PSL_{d+1}(F) \cdot  PGL_{d+1}(\calO)$, which is a normal subgroup of index $d+1$ in $G$, but not by that of $G$.
Still, $\tau$ induces a coloring on the oriented edges of $B$: $\tau([L_1],[L_2])= \tau([L_1])-\tau([L_2])(\mbox{mod}(d+1))$, and this coloring of the edges is preserved by $G$. The coloring of the (oriented) edges defines $d$ "Hecke operators" $A_1, \cdots, A_d$ as follows:
For $f \in L^2(B(0))$,
$$A_i(f)(x) = \sum\{f(y) | (x,y) \in B(1), \tau((x,y))=i\}.$$
The operators $A_i$ are normal (though not self adjoint) and commute with each other, hence can be diagonalized simultaneously.

Every cocompact discrete subgroup $\Gamma$ of $G$ acts on $B$ and $X=\Gamma\backslash B$ is a finite complex. For simplicity we will assume that for every vertex $x$ of $B$ and every $1\neq\gamma \in \Gamma$, $\mbox{dist}(\gamma x, x) > 2$. This ensures that there are no ramifications and $\Gamma \backslash B$ is indeed a simplicial complex. This can always be achieved by replacing $\Gamma$ by a finite index subgroup (and by a congruence one if $\Gamma$ is arithmetic).

Since $G$ (and hence $\Gamma$) preserves the coloring of the oriented edges, the operator $A_i$ is well defined also on $L^2(X(0))$. In~\cite{LSV1}, the finite complex $X$ is called Ramanujan if the "non trivial spectrum" of ($A_1, \cdots, A_d$) on $L^2(X(0))$ (which is a subset of $\C^d$) is contained in the spectrum of ($A_1, \cdots, A_d$) acting on $L^2(B(0))$ - see there for exact definitions. The trivial spectrum consists, in general, of at most $d$ eigenvalues. More precisely, if $\Gamma G_0$ is of index $r$ in $G$ then $\Gamma \backslash B$ has $r$ "trivial eigenvalues" (see~\cite[Section 2.3 and Proposition 6.7]{LSV1}. For example, for $d=1$, it has either two trivial eigenvalues, if $\Gamma \backslash B$ is a bipartite graph, or just one, if it is not. Similarly, if $\Gamma \leq G_0$, there are $d+1$ trivial eigenvalues or just one if $\Gamma G_0 = G$. To avoid the trouble of handling the trivial eigenvalues, we will work all the time with "non-partite Ramanujan complexes", i.e., those obtained by lattices $\Gamma$ with $\Gamma G_0 = G$. By~\cite[Theorem 7.1]{LSV1} there are infinitely many such finite quotients $X= \Gamma \backslash B$.

What is important for us here is the following:
$A_1+\cdots +A_d$ is acting on $L^2(X(0))$ exactly as the adjacency matrix of the graph $X^{(1)}$ which is a $k$-regular graph with $k\sim q^{\frac{(d+1)^2}{4}}$.
From the definition of Ramanujan complexes we deduce~\cite{LSV1}:

\begin{cor}~\label{cor-bounding-lambda-in-quotient-Ram-complex}
If $X$ is a non-partite Ramanujan complex, a quotient of $B=\tilde{A}_d(F)$, $d \geq 1$, as above, then the second largest eigenvalue of the adjacency matrix of $X^{(1)}$ is bounded from above by ${d+1 \choose  \lfloor\frac{d+1}{2}\rfloor}\sqrt{k} \leq (d+1)^{d+1}q^{\frac{(d+1)^2}{8}}$ and thus $\lambda_1(X^{(1)}) \geq k-(d+1)^{d+1}\sqrt{k}$, (see Proposition~\ref{prop-cheeger}) so as graphs, for $q$ large w.r.t. $d$, $X^{(1)}$ is almost a Ramanujan graph. If $d=2$ an improved bound is known: $\lambda_1(X^{(1)}) \geq k-6\sqrt{k}$.
\end{cor}

From spectral point of view, Ramanujan complexes are excellent high dimensional expanders, but they are not necessarily "coboundary expanders" in the sense of Definition~\ref{def-coboundary-cocycle-exp}. Indeed, if $\epsilon_i(X) > 0$ then $H^i(X,\F_2) = 0$ but we have the following.

\begin{proposition}~\label{prop-non-trivial-first-cohomology} For every $d \geq 1$ and every prime power $q$, there are infinitely many Ramanujan complexes $X$, quotients of $\tilde{A}_d(\F_q((t)))$, with $H^1(X,\F_2) \neq 0$.
\end{proposition}
%\tnote{dimension of $H^1$ with normal subgroups}

\begin{proof} As shown in~\cite{LSV1}, for $F=\F_q((t))$, and for every fixed $d$, there is an arithmetic lattice $\Gamma_0 < PGL_{d+1}(F)$ with infinitely many congruence normal subgroups $\Gamma_i \lhd \Gamma_0$ such that $\Gamma_0/\Gamma_i \simeq PSL_{d+1}(q^{s_i})$ with $s_i \rightarrow \infty$, and $\Gamma_i \backslash \tilde{A}_d(F)$ is a Ramanujan complex.

Let $S_2$ be the $2$-Sylow subgroup of $PSL_{d+1}(q^{s_i})$ and $\tilde{\Gamma}_i$ its preimage in $\Gamma_0$. Then, $X=\tilde{\Gamma}_i \backslash B$, being a quotient of a Ramanujan complex, is also Ramanujan. But,
$$\tilde{\Gamma}_i/([\tilde{\Gamma}_i,\tilde{\Gamma}_i]\tilde{\Gamma}_i^2) \twoheadrightarrow S_2/([S_2,S_2]S_2^2) \neq \{0\}.$$

As $B$ is contractible,
$$H^1(X,\F_2)=H^1(\tilde{\Gamma}_i \backslash B, \F_2) = H^1(\tilde{\Gamma}_i, \F_2) = \tilde{\Gamma}_i/([\tilde{\Gamma}_i,\tilde{\Gamma}_i]\tilde{\Gamma}_i^2) \neq \{0\},$$ and the proposition is proved.
\end{proof}

A similar result hold also for the second cohomology group.

\begin{proposition}\label{prop-non-trivial-second-cohomology} For every $d \geq 2$ and every prime power $q$, there exist Ramanujan complexes $X$, quotients of $\tilde{A}_d(\F_q((t)))$ with $H^2(X,\F_2) \neq 0$.
\end{proposition}

We will prove first a purely group theoretic result which may be of independent interest.

\begin{proposition}\label{prop-H2-group} Let $\Gamma$ be a discrete group, $\hat{\Gamma}$ its profinite completion and $\Gamma_{\hat{p}}$ its pro-$p$ completion. (We do not assume that $\Gamma$ is residually finite nor residually-$p$, so $\Gamma$ may not inject into $\hat{\Gamma}$ or $\Gamma_{\hat{p}}$). Then
\begin{enumerate}
\item\label{item-prop-H2-first} If $H^2(\hat{\Gamma}, \F_p) \neq 0$ then $H^2(\Gamma, \F_p) \neq 0$.
\item\label{item-prop-H2-second} If $H^2(\Gamma_{\hat{p}}, \F_p) \neq 0$ then $H^2(\Gamma, \F_p) \neq 0$.
\end{enumerate}
\end{proposition}

\begin{proof} As it is well known, for every discrete or profinite group $G$, $H^2(G,\F_p)$ classifies equivalent classes of central (continuous) extensions $E$ of $G$ by $\F_p$ (~\cite[Theorem 6.8.4]{RilvesZalesskii})
\begin{equation}\label{eqn-H2-first}
1 \rightarrow  \F_p \rightarrow E \rightarrow G \rightarrow 1.
\end{equation}
Now, $H^2(G,\F_p)=0$ means that every central extension as (\ref{eqn-H2-first}) splits.

Assume there is a non-splitting extension
\begin{equation}\label{eqn-H2-second}
1 \rightarrow  \F_p \rightarrow E  \xrightarrow{\eta} \hat{\Gamma} \rightarrow 1.
\end{equation}
Let $E_0 = \{(a,b) \in \Gamma \times E | \mbox{ } i(a) = \eta(b)\}$ where $i: \Gamma \rightarrow \hat{\Gamma}$ is the natural map from $\Gamma$ to its profinite completion. This gives rise to an extension
\begin{equation}\label{eqn-H2-third}
1 \rightarrow  \F_p \rightarrow E_0  \xrightarrow{\pi} \Gamma \rightarrow 1.
\end{equation}
where $\pi(a,b)=a$ for $(a,b) \in E_0$. Indeed, $\pi$ is an epimorphism as for every $a \in \Gamma$, there exists $b \in E$ with $\eta(b) = i(a)$ since $\eta$ is an epimorphism from $E$ onto $\hat{\Gamma}$. Moreover, $\mbox{ker}(\pi)=\{(a,b) \in E_0 | \mbox{ } a=e\} = \{(e,b) | \mbox { } \eta(b) = e_{\hat{\Gamma}}\}\simeq \F_p$.
We claim that (\ref{eqn-H2-third}) is not a splitting sequence. Otherwise, there exists $\pi':\Gamma \rightarrow E_0$ with $\pi \circ \pi' = \mbox{id}_{\Gamma}$.
Thus, there exists $\widehat{\pi'}:\hat{\Gamma} \rightarrow \hat{E_0}$. But, it it easy to see that $\hat{E_0}\simeq E$ and such $\widehat{\pi'}$ would split (\ref{eqn-H2-second}), a contradiction. This proves (\ref{item-prop-H2-first}). The proof of (\ref{item-prop-H2-second}) is similar, replacing profinite completion by pro-$p$ completion.
\end{proof}

We can now prove Proposition~\ref{prop-non-trivial-second-cohomology}:
\begin{proof}
Let $\Gamma=\tilde{\Gamma_i}$ be as in the proof of Proposition~\ref{prop-non-trivial-first-cohomology}. As shown there $\Gamma$ has a non-trivial finite quotient of $2$-power order. Thus, its pro-$2$ completion is not the trivial group. It is also not a free pro-$p$ group since $\Gamma$ has property(T) (note $d+1 \geq 3$) and hence $\Gamma/[\Gamma,\Gamma]$ is finite. Thus a minimal presentation of the finitely generated pro-$2$ group $\Gamma_{\hat{2}}$ requires at least one relation and hence by~\cite[Theorem 7.8.3]{RilvesZalesskii} $H^2(\Gamma_{\hat{2}},\F_2) \neq 0$. We can apply now Proposition~\ref{prop-H2-group} to deduce that $H^2(\Gamma,\F_2) \neq 0$. As in the proof of Proposition~\ref{prop-non-trivial-first-cohomology}, we can conclude that $H^2(X,\F_2) \neq 0$.
\end{proof}

%Proposition~\ref{prop-non-trivial-first-cohomology} and Proposition~\ref{prop-non-trivial-second-cohomology} prove Proposition~\ref{prop-non-trivial-first-second-cohomology} promised in the introduction.
We formulate Proposition~\ref{prop-non-trivial-first-cohomology} and Proposition~\ref{prop-non-trivial-second-cohomology} in the way which is most interesting for us, i.e., showing that Ramanujan complexes are not necessarily coboundary expanders. But, in fact, the proofs show that
for every cocompact lattice $\Gamma$ in $PGL_{d+1}(F)$, $d\geq 2$, has a finite index subgroup $\Gamma'$ with $H^1(\Gamma',\F_2) \neq 0$ and $H^2(\Gamma',\F_2) \neq 0$. We do not know if analogues results are valid for $H^i$, for $i \geq 3$ (and $d \geq i$). Our proofs of Proposition~\ref{prop-non-trivial-first-cohomology} and Proposition~\ref{prop-non-trivial-second-cohomology} use the explicit group theoretic interpretation of the first and second cohomology groups. No such explicit interpretation is known for $H^i$, $i \geq 3$.

%On the positive side let us mention that for $d \geq 3$ a quantitative from of property$(T)$ implies that arbitrary quotients of $B$ are good expanders from a spectral point of view (and one does not really need the full strength of the Ramanujan property). For them we have.
%
%\begin{proposition}~\label{prop-propery-T-esimate-on-lambda-1} For $d \geq 2$ denote $k=\sum_{i=1}^{d}{d \choose i}_q$, i.e., $k$ is the number of non-trivial proper subspaces of $\F_q^{d+1}$, which is the degree of $B^{(1)}$, the $1$-skeleton of the Bruat-Tits building $B$ of $PGL_{d+1}(F)$. Then there exists a constant $M=M(d)$, such that for every finite quotient $X$ of $B$, $\lambda_1(X^{(1)}) \geq k-M(d)\frac{k}{\sqrt{q}}$.
%\end{proposition}
%\tnote{check+ref+was $\lambda_1$ defined?}

\section{From Isoperimetric inequalities to topological expanders}
\label{section-main-proof}
In this section we show that the isoperimetric inequalities of Theorem~\ref{thm-isoperimetric -inequalities} imply Theorem~\ref{thm-main}. The connection is via (an extended version of) Gromov's Theorem, Theorem~\ref{thm-Gromov-criteria-for-top-exp}.

So, we fix now a very large prime power $q$ and write $F=\F_q((t))$, $B=A_3(F)$ the $3$-dimensional Bruhat-Tits building associated with $PGL_4(F)$, $X$ a non-partite Ramanujan quotient of $B$ and $Y=X^{(2)}$, the $2$-skeleton of $X$. In $\cite{LSV2}$, it was shown that there are infinitely many such $X$'s with $|X| \rightarrow \infty$. Our goal is to show that the $2$-dimensional simplicial complex $Y$ has the $\epsilon$-topological overlapping property for some $\epsilon> 0$, depending maybe on $q$, but not on $X$ or $Y$. This will prove Theorem~\ref{thm-main} and answers Gromov's question in the affirmative as every vertex of $Y$ is contained in at most $O(q^5)$ $2$-cells.

To this end, we should show now that $Y$ satisfies the assumption of Theorem~\ref{thm-Gromov-criteria-for-top-exp}. Here $d=2$ and we have to show that $\mu_i(Y)$, $i=0,1$ are bounded from above and $syst^i(X)$, $i=0,1$ are bounded from below.

Let us start with the systole. As $Y^{(1)}=X^{(1)}$ is connected, $H^0(Y,\F_2) = 0$ and so $syst^0(Y)=\infty$ and this case is trivial. The argument for $syst^1(Y)$ is more involved. Here, it is possible that $H^1(Y,\F_2)=H^1(X,\F_2)$ is non zero (see Proposition~\ref{prop-non-trivial-first-cohomology}). So, let $\alpha \in Z^1(Y,\F_2) \setminus B^1(Y,\F_2)$. If $\alpha$ is not locally minimal, then by Proposition~\ref{prop-loc-min-properties}(\ref{item-one-prop-loc-min-properties}) we can replace it by a locally minimal $\alpha'$ with $||\alpha'|| \leq ||\alpha||$ and $\alpha' \equiv \alpha (\mbox{mod }B^i)$, so $\alpha'$ is also in $Z^1(Y,\F_2) \setminus B^1(Y,\F_2)$. Thus, to prove the lower bound on $syst^1(Y)$, we can assume $\alpha$ is locally minimal and we claim now that $||\alpha|| > \eta_1$, for the $\eta_1$ of Theorem~\ref{thm-isoperimetric -inequalities}. If not, then by that theorem, $||\delta_1(\alpha)|| \geq \epsilon_1||\alpha||$. But, $\alpha \in Z^1$, so $\delta_1(\alpha) =0$ and hence $\alpha =0$, in contradiction to the assumption that $\alpha \notin B^1$.

We now turn to prove upper bounds on the filling norms $\mu_0$ and $\mu_1$ of $Y$. Let $\beta \in B^{i+1}(Y,\F_2)$, $i=0$ or $i=1$, so $\beta=\delta_i(\alpha)$ for some $\alpha \in C^i(X,\F_2)$. We claim that one can choose such $\alpha$ with
\begin{equation}\label{equation-proof-main-thm}
||\alpha|| \leq \mu_i||\beta||, \mbox{  } \mu_i = \mbox{max}(\frac{1}{\eta_{i+1}},\frac{2-i}{i+2}m(i) )
\end{equation}
where $\eta_{i+1}$ is the one from Theorem~\ref{thm-isoperimetric -inequalities} and $m(i)$ is the one from Proposition~\ref{prop-loc-min-properties}, i.e., the number of $i$-cells containing a vertex. To see this, assume first $||\beta|| > \eta_{i+1}$. As we always have $||\alpha|| \leq 1$, (\ref{equation-proof-main-thm}) clearly holds, so assume $||\beta|| \leq \eta_{i+1}$. Apply Proposition~\ref{prop-loc-min-properties}(\ref{item-one-prop-loc-min-properties}) for $Y$ whose dimension is $2$ and for $i+1$: we can replace $\beta$ by a locally minimal $\beta'$ with $\beta' \equiv \beta (\mbox{ mod } B^{i+1})$, so $\beta'$ is also a coboundary, $||\beta'|| \leq ||\beta||$, so $||\beta'|| \leq \eta_{i+1}$ and furthermore $\beta' = \beta  +\delta_i(\gamma)$ where $\gamma \in C^{i}(X, \F_2)$ with $||\gamma|| \leq c_i||\beta|| $ when $c_i = \frac{2-i}{i+2}m(i)$. Here $m(i)$ is $1$ when $i=0$ and $O(q^4)$ for $i=1$.

As $\beta'$ is locally minimal in $C^{i+1}(Y,\F_2) =  C^{i+1}(X,\F_2)$ \footnote{Note that the norms in $C^{i+1}(Y,\F_2)$ and $C^{i+1}(X,\F_2)$ are the same since every $2$-cells of $X$ is contained in exactly $(q+1)$ $3$-cells of $X$ } and $||\beta'|| \leq \eta_{i+1}$, Theorem~\ref{thm-isoperimetric -inequalities} implies that $||\delta_{i+1}(\beta')|| \geq \epsilon_{i+1} ||\beta'||$. But, $\beta' \in B^{i+1}(Y,\F_2) = B^{i+1}(X,\F_2) \subseteq Z^{i+1}(X,\F_2)$ so $\delta_{i+1}(\beta') = 0$ and hence $\beta'=0$. Thus, $\beta= \delta_i(\gamma)$ and again (\ref{equation-proof-main-thm}) is valid and Theorem~\ref{thm-main} is proved.

\begin{remark} The reader should note that in order to prove that $\mu_1$ is bounded from above, we have used $\delta_2: C^2(X,\F_2) \rightarrow C^3(X,\F_2)$, i.e., we have used the $3$-dimensional complex $X$ even though our result is for the $2$-dimensional complex $Y$. This is the crucial point which enables us to prove Theorem~\ref{thm-main} for $Y$, while we do not know the topological overlapping property for $2$-dimensional Ramanujan complexes.
\end{remark}

We finally note that the method of proof gives also a systolic inequity for $X$ as above:
\begin{cor}
Let $X$ be a non-partite Ramanujan complex of dimension $3$ as above. Then for $i=0,1,2$, $syst^i(X) \geq \nu_i$ for some constants $\nu_i > 0$.
\end{cor}
\begin{proof}
For $i=0$, $H^0(X) =0$ and there is nothing to prove and for $i=1$, $syst^1(X)=syst^1(Y)$ where $Y = X^{(2)}$ and this was proved above. For $i=2$, we can argue in a similar way as before: if $\alpha \in Z^2 \setminus B^2$ (such $\alpha$ can exist - see Proposition~\ref{prop-non-trivial-second-cohomology}) we can replace it by a locally minimal one $\alpha'$ and argue as before to deduce that $||\alpha'|| \geq \eta_2$.
\end{proof}

\section{Expansion of $1$-cochains in $2$-dimensional Ramanujan complexes}
\label{section-two-dim-proof}
In this section we prove Theorem~\ref{thm-two-dim-coboundary-exp}. We note that in this case every vertex (edge) is in a constant number of triangles so the norm (on vertices or on edges) is the normalized counting norm. It will be easier therefore to work here simply with the counting norm $|\alpha|$ and in the end of the proof "to translate" the result to $||\alpha||$.

\subsection{Proof of Theorem~\ref{thm-two-dim-coboundary-exp}}
So $X$ is a Ramanujan complex of dimension $2$. Every vertex $v$ has degree $Q=2(q^2+q+1)$ and the link $X_v$ at any vertex $v$ is the "lines versus points" graph of the projective plane $\P^2(\F_q)$, which is a $(q+1)$-regular bipartite graph on $2(q^2+q+1)$ points. The cochain $\alpha$ can be thought of as a set of edges of $X$ such that $|\alpha_v| \leq  \frac{Q}{2}$ for every $v$, since $\alpha$ is locally minimal (see Proposition~\ref{prop-loc-min-properties} (\ref{item-two-prop-loc-min-properties})).

\begin{lemma}~\label{lemma-traingles-counting-in-dim-two}
For $i=0,1,2,3$ denote by $t_i$, the number of triangles of $X$ which contain exactly $i$ edges from $\alpha$. Then,
\begin{enumerate}
\item~\label{item-one-lemma-traingles-counting-in-dim-two} $t_1+2t_2+3t_3=(q+1)|\alpha|$.
\item~\label{item-two-lemma-traingles-counting-in-dim-two} $|\delta_1(\alpha)| = t_1+t_3$.
\item~\label{item-three-lemma-traingles-counting-in-dim-two} $\sum_{v \in X(0)}|E_{X_v}(\alpha_v, \overline{\alpha_v})|=2t_1+2t_2$.
\end{enumerate}
Here we consider $\alpha_v$, which is the set of edges of $\alpha$ touching $v$, as a set of vertices of the link $X_v$. By $\overline{\alpha_v}$ we denote its complement there and $E_{X_v}(\alpha_v, \overline{\alpha_v})$ the set of edges from $\alpha_v$ to $\overline{\alpha_v}$.
\end{lemma}

\begin{proof} For (\ref{item-one-lemma-traingles-counting-in-dim-two}) we recall that every edge lies on $q+1$ triangles and a triangle which contributes to $t_i$ contains $i$ edges from $\alpha$. Part (\ref{item-two-lemma-traingles-counting-in-dim-two}) is simply the definition of $\delta_1(\alpha)$, which is the set of all triangles containing an odd number of edges from $\alpha$. For (\ref{item-three-lemma-traingles-counting-in-dim-two}) we argue as follows.

If $\triangle =\{v_0,v_1,v_2\}$ is a triangle of $X$, then it contributes an edge at $X_{v_k}$ ($\{ v_k\} = \{v_i,v_j,v_k\}\backslash \{v_i,v_j\} )$. This is the edge between $e_{i,k}=(v_i,v_k)$ and $e_{j,k}=(v_j,v_k)$ when we consider $e_{i,k}$ and $e_{j,k}$ as vertices of $X_{v_k}$. This edge will be in $E_{X_{v_k}}(\alpha_{v_k}, \overline{\alpha_{v_k}})$ if and only if exactly one of $\{e_{i,k},e_{j,k}\}$ is in $\alpha$. A case by case analysis of the four possibilities shows that if $\triangle$ has either $0$ or $3$ edges from $\alpha$ then $\triangle$ does not contribute anything to the left hand sum. On the other hand, if it has either $1$ or $2$ edges, it contributes $2$ to the sum. This proves the lemma.
\end{proof}

Fix now a small $\epsilon>0$ to be determined later and define:

\begin{definition} A vertex $v$ of $X$ is called {\em thin} w.r.t. $\alpha$ if $|\alpha_v| < (1-\epsilon)\frac{Q}{2}$ and {\em thick} otherwise (recall that by our local minimality assumption, $|\alpha_v| \leq \frac{Q}{2}$ for every $v$).

Denote
\begin{itemize}
\item $W=\{v \in V=X(0) | \mbox{ } \exists e \in \alpha \mbox{ with } v \in e\}$.
\item $R = \{v \in W | \mbox{ } v \mbox{ thin} \}$.
\item $S = \{v \in W | \mbox{ } v \mbox{ thick} \} = W \setminus R $.
\end{itemize}

Let $r=\sum_{v \in R}|\alpha_v|$ and $s=\sum_{v \in S}|\alpha_v|$.
\end{definition}

\begin{lemma}~\label{lemma-$r+s$}
$r+s = 2|\alpha|$
\end{lemma}
\begin{proof} Every edge in $\alpha$ contributes $2$ to the left hand side.
\end{proof}

\begin{lemma}~\label{lemma-edges-exiting-alpha-v}
\begin{enumerate}
\item~\label{item-one-lemma-edges-exiting-alpha-v} For every $v \in V$, $|E_{X_v}(\alpha_v, \overline{\alpha_v})| \geq \frac{1}{2}(q+1-\sqrt{q})|\alpha_v|$.
\item~\label{item-two-lemma-edges-exiting-alpha-v} If $v$ is thin, then $|E_{X_v}(\alpha_v, \overline{\alpha_v})| \geq \frac{(1+\epsilon)}{2}(q+1-\sqrt{q})|\alpha_v|$.
\end{enumerate}
\end{lemma}

\begin{proof} As mentioned in Section~\ref{subsection-spherical-buildings}, the link $X_v$ is the "line versus points" graph of the projective plane. It is a $(q+1)$-regular graph whose eigenvalues are $\pm(q+1)$ and $\pm \sqrt{q}$. Hence,
$\lambda_1(X_v) = (q+1)-\sqrt{q}$. Part~\ref{item-one-lemma-edges-exiting-alpha-v} now follows from Proposition~\ref{prop-cheeger}, and similarly
part~\ref{item-two-lemma-edges-exiting-alpha-v}.
\end{proof}

We can deduce

\begin{lemma}~\label{lemma-2t-1+2t-2}
$2t_1 + 2t_2 \geq (q+1-\sqrt{q})|\alpha| + \frac{\epsilon}{2} (q+1-\sqrt{q}) r$.
\end{lemma}

\begin{proof}
\begin{eqnarray}
2t_1 + 2t_2 = \sum_{v \in W}E_{X_v}(\alpha_v, \overline{\alpha_v})
& =    &  \sum_{v \in R}E_{X_v}(\alpha_v, \overline{\alpha_v}) + \sum_{v \in S}E_{X_v}(\alpha_v, \overline{\alpha_v}) \\
& \geq &  \sum_{v \in R}\frac{(1+\epsilon)}{2}(q+1-\sqrt{q})|\alpha_v| + \sum_{v \in S}\frac{1}{2}(q+1-\sqrt{q})|\alpha_v| \\
& =    &  \frac{(1+\epsilon)}{2}(q+1-\sqrt{q})r + \frac{1}{2}(q+1-\sqrt{q})s \\
& =    &  \frac{1}{2}(q+1-\sqrt{q})(r+s)+\frac{\epsilon}{2} (q+1-\sqrt{q})r \\
& =    &(q+1-\sqrt{q})|\alpha| + \frac{\epsilon}{2}(q+1-\sqrt{q})r
\end{eqnarray}
In the first equation we have used Lemma~\ref{lemma-traingles-counting-in-dim-two}, part (\ref{item-three-lemma-traingles-counting-in-dim-two}) and in the last one Lemma~\ref{lemma-$r+s$}. The inequality follows from Lemma~\ref{lemma-edges-exiting-alpha-v}.
\end{proof}
%\tnote{check, should it be epsilon divided by two?}

\begin{lemma}\label{lemma-t-1-minus-3t-3}
$t_1 - 3t_3 \geq \frac{\epsilon}{2} (q+1-\sqrt{q})r - \sqrt{q}|\alpha|$.
\end{lemma}

\begin{proof} Subtract equation (\ref{item-one-lemma-traingles-counting-in-dim-two}) in Lemma~\ref{lemma-traingles-counting-in-dim-two} from the equation obtained in Lemma~\ref{lemma-2t-1+2t-2}.
\end{proof}

Our goal now is to show that $r$, the contribution of the thin edges is at least some fixed fraction of $|\alpha|$. This will prove that for $q$ large enough $t_1 \geq cq|\alpha|$ and this will give the theorem. Up to now we have used only the local structure of $X$, the links. Now we will use the global structure, the fact that its $1$-skeleton is almost a Ramanujan graph.

\begin{lemma}\label{lemma-bound-on-number-of-edges-within-thick-vertices} The total number of edges in $X^{(1)}$ between the thick vertices is bounded as follows:
$$|E_{X^{(1)}}(S)| \leq |\alpha|(\frac{1}{(1-\epsilon)^2(1+\epsilon_0)} + \frac{12q}{(1-\epsilon)Q}). $$
\end{lemma}

\begin{proof}
Recall, that by Corollary~\ref{cor-bounding-lambda-in-quotient-Ram-complex}, the second largest eigenvalue of the adjacency matrix of $X^{(1)}$ is bounded from above by $6q$.
So $\lambda_1(X^{(1)}) \geq Q-6q=2{q^2}-4q+1$. Note now that every vertex in $S$ touches at least $(1-\epsilon)\frac{Q}{2}$ edges of $\alpha$, hence $|S| \leq \frac{2|\alpha|}{(1-\epsilon)\frac{Q}{2}}=\frac{4|\alpha|}{(1-\epsilon)Q}$. Proposition~\ref{prop-cheeger} implies therefore (when $|X(0)|=n$)

\begin{eqnarray}
|E(S)| & \leq & \frac{1}{2}(Q-\frac{|\overline{S}|}{n}\lambda_1(X^{(1)}))|S| \\
     & \leq & \frac{1}{2}(Q-\frac{|\overline{S}|}{n}(Q-6q))|S| =  \frac{1}{2}(Q(1-\frac{|\overline{S}|}{n})+6q \frac{|\overline{S}|}{n})|S|\\
     & \leq & \frac{1}{2}(Q \frac{|S|}{n}+6q)|S| \leq \frac{1}{2}(\frac{4|\alpha|}{(1-\epsilon)n}+6q)|S|
\end{eqnarray}

Note now that the assumption $||\alpha|| \leq \frac{1}{4(1+\epsilon_0)}$ means $|\alpha| \leq \frac{Qn}{8(1+\epsilon_0)}$ and hence,
$$|E(S)| \leq (\frac{2Q}{8(1-\epsilon)(1+\epsilon_0)}+3q)\frac{4|\alpha|}{(1-\epsilon)Q}=|\alpha|(\frac{1}{(1-\epsilon)^2(1+\epsilon_0)} + \frac{12q}{(1-\epsilon)Q}). $$
\end{proof}

\begin{proof}(of Theorem~\ref{thm-two-dim-coboundary-exp})
We can now finish the proof of Theorem~\ref{thm-two-dim-coboundary-exp}. Choose $\epsilon > 0$ such that $\frac{1}{(1-\epsilon)^2(1+\epsilon_0)} < 1$
and then assume that $q$ is sufficiently large such that $\frac{1}{(1-\epsilon)^2(1+\epsilon_0)} + \frac{12q}{(1-\epsilon)Q} < 1-\xi < 1$, for some $\xi > 0$. This now means by Lemma~\ref{lemma-bound-on-number-of-edges-within-thick-vertices} that at most $(1-\xi)$ of the edges in $\alpha$ are between two thick vertices, namely, for at least $\xi|\alpha|$ edges, one of their endpoints is thin. This implies that $r \geq \xi|\alpha|$. Plugging this in Lemma~\ref{lemma-t-1-minus-3t-3}, we get $t_1  \geq  (\frac{\epsilon}{2}(q+1-\sqrt{q})\xi - \sqrt{q})|\alpha|$. Again, if $q$ is large enough this means that $|\delta_1(\alpha)| \geq t_1 \geq \epsilon_1 q |\alpha|$.

Now for $\beta \in C^2(X,\F_2)$, $||\beta|| =\frac{|\beta|}{|X(2)|}$. For $\alpha \in C^1(X,\F_2)$, $||\alpha||= \frac{(q+1)|\alpha|}{2|X(2)|}$.
Thus, $||\delta_1(\alpha)||=\frac{\delta_1(\alpha)}{|X(2)|} \geq \frac{\epsilon_1 q |\alpha|}{|X(2)|} =  \frac{\epsilon_1 q \cdot 3|X(2)| \cdot ||\alpha||}{(q+1)|X(2)|} \geq \epsilon_2||\alpha||$ for $\epsilon_2 = 3\frac{q}{q+1} \epsilon_1 \geq 2\epsilon_1$.
Theorem~\ref{thm-two-dim-coboundary-exp} is now proved.
\end{proof}

The proof is effective. One can estimate $\epsilon_2$ and how large should be $q$, in term of $\epsilon_0$. It is independent of $q$ provided $q$ is large enough.
%Furthermore, we used the global structure only in the last argument. It would work (with a price of having $q$ larger) also if we use only the weaker bound on $\lambda_1(X^{(1)})$ obtained from property $(T)$ (see Proposition~\ref{prop-propery-T-esimate-on-lambda-1}) rather than the strong bound from Corollary~\ref{cor-bounding-lambda-in-quotient-Ram-complex}. So, Theorem~\ref{thm-two-dim-coboundary-exp} is valid for all the finite quotients of $B=\tilde{A}_2(F)$ provided $q \gg 0$.
%\tnote{check!}

Let us mention that along the way we have proved two facts which are worth formulating separately.

\begin{cor}\label{cor-from-2-dim-result} In the notations and assumptions as above. For every $\epsilon > 0$, if $q \geq q(\epsilon) \gg 0$, then we have:
%in particular, $q \gg 0$: There exists $\rho > 0$ such that:
\begin{enumerate}
\item If $\alpha \in B^1(X,\F_2)$ is a locally minimal coboundary with $||\alpha|| < \frac{1}{4(1+\epsilon)}$ then $\alpha = 0$.
\item If $\alpha \in Z^1(X,\F_2) \setminus B^1(X,\F_2)$, then $||\alpha|| > \frac{1}{4(1+\epsilon)}$, In particular, every representative of a non-trivial cohomology class has linear size support.
\end{enumerate}
\end{cor}

This is the systolic inequality promised in Corollary~\ref{cor-systolic-two-dim} of the introduction. Note that by Proposition~\ref{prop-non-trivial-first-cohomology}, there are indeed cases with $H^1(X,\F_2) \neq \{0\}$ so the second item of Corollary~\ref{cor-from-2-dim-result} is a non-vacuous systolic result. Such results are of potential interest also for quantum error-correcting codes (see \cite{LubotzkyGuth},\cite{Zemor} and the references therein).

\section{Expansion of $i$-cochains in $3$-dimensional Ramanujan complexes}
\label{section-one-cochain-three-dim-proof}
In this section we prove Theorem~\ref{thm-isoperimetric -inequalities} for the cases $i=0$ and $i=1$. Let us start with the easier case - vertex expansion.
\subsection{Proof of Theorem~\ref{thm-isoperimetric -inequalities} for the case $i=0$}
The case $i=0$ of Theorem~\ref{thm-isoperimetric -inequalities} is nothing more than the standard result asserting that $X^{(1)}$ - the $1$-skeleton of $X$ is an expander graph. But, some care is needed here since the edges of $X^{(1)}$, when considered as edges of a graph get equal weights, but when considered as edges of the $3$-dimensional complex $X$, have different weights. In fact, a black edge, i.e., one corresponding to the $1$ or $3$ dimensional subspace in $\F_q^4$, when we look at the links of its vertices, has $\Theta$-times the weight of a white edge (an edge which corresponds to a $2$-dimensional subspace of $\F_q^4$) when $\Theta = \frac{c_1^b}{c_1^w}=\frac{{3 \choose 1}_q \cdot {2 \choose 1}_q}{{2 \choose 1}_q \cdot {2 \choose 1}_q} =\frac{q^2+q+1}{q+1} \approx q$ since a $1$ or $3$ dimensional subspace is contained in $c_1^b = {3 \choose 1}_q \cdot {2 \choose 1}_q$ maximal flags in $\F_q^4$, while a $2$-dimensional subspace only in $c_1^w = {2 \choose 1}_q \cdot {2 \choose 1}_q$ such flags.

Let $\alpha \in C^0(X,\F_2)$ be a locally minimal $0$-cochain of $X$, i.e., a minimal cochain (see Section~\ref{subsection-notions-of-minimality}). So, $\alpha$ is a subset of the $X(0)$ - the set of vertices of $X$ - containing at most half of the vertices (since all the vertices have the same weight). By Corollary~\ref{cor-bounding-lambda-in-quotient-Ram-complex}, $\lambda_1(X^{(1)}) \geq k - 4^4\sqrt{k}$ where $k$ is the degree of the $k$-regular graph $X^{(1)}$, so $k \approx q^4$. Now, Proposition~\ref{prop-cheeger} implies that

\begin{eqnarray}\label{eqn:case-i-0-first}
|\delta_0(\alpha)| = |E(\alpha, \bar{\alpha})| \geq \frac{|\alpha||\bar{\alpha}|}{|X^{(0)}|}(q^4 - 4^4q^2) \geq \frac{1}{2}(q^4 - 4^4q^2)|\alpha|.
\end{eqnarray}

In terms of norms:

\begin{eqnarray}\label{eqn:case-i-0-second}
||\alpha|| = |\alpha| \cdot \frac{c_0}{{4 \choose 1}|X(3)|}
\end{eqnarray}

where $c_0$ is the number of $3$-cells of $X$ containing a vertex $v$. This number is independent of $v$, equal to the number of maximal flags in $\F_q^4$ and it is approximately $q^6$. On the other hand, if $\beta:=\delta_0(\alpha)$, $\beta = \beta^b + \beta^w$ where $\beta^b$ (resp. $\beta^w$) is the set of black (resp. white) edges of $\beta$, then

\begin{eqnarray}\label{eqn:case-i-0-third}
||\delta_0(\alpha)|| = ||\beta|| = \frac{1}{{4 \choose 2}|X(3)|}(c_1^b|\beta^b| + c_1^w|\beta^w|) = \frac{(q+1)^2}{{4 \choose 2}|X(3)|}(\Theta|\beta^b| + |\beta^w|) \geq \frac{(q+1)^2}{{4 \choose 2}|X(3)|} |\beta|=\frac{(q+1)^2}{{4 \choose 2}|X(3)|} |\delta_0(\alpha)|
\end{eqnarray}

%where $c_1$ is the minimal number of $3$-cells containing an edge $e$. Here, different edges give different number, but one can easily check that it is at most $q^2$ (for white edge, for black it is at least $q^3$).

Combining, (\ref{eqn:case-i-0-third}), (\ref{eqn:case-i-0-first}) and (\ref{eqn:case-i-0-second}) we deduce:

\begin{eqnarray}\label{eqn:case-i-0-forth}
||\delta_0(\alpha)|| \geq  \frac{(q+1)^2}{{4 \choose 2}|X(3)|} |\delta_0(\alpha)| \geq \frac{(q+1)^2}{{4 \choose 2}|X(3)|} \frac{1}{2} (q^4 - 4^4q^2)|\alpha| =  \frac{(q+1)^2}{{4 \choose 2}|X(3)|} \frac{1}{2} (q^4 - 4^4q^2) \frac{{4 \choose 1} |X(3)|}{c_0}||\alpha||\geq \epsilon_0 ||\alpha||
\end{eqnarray}

Case $i=0$ of Theorem~\ref{thm-isoperimetric -inequalities} is proven, with $\eta_0=1$ and $\epsilon_0$ independent of $q$, since $c_0 \approx q^6$.

\subsection{Proof of Theorem~\ref{thm-isoperimetric -inequalities} for the case $i=1$}
The main idea of the proof is similar to the one that was shown in Section~\ref{section-two-dim-proof} for $1$-cochains in Ramanujan complexes of dimension $2$, but here edges have different weights, so more care is needed. Let $\alpha \in C^1(X,\F_2)$ be a $1$-cochain of a $3$-dimensional non-partite Ramanujan Complex $X$. The cochain $\alpha$ is a collection of edges of two types: black and white as described in Section~\ref{section-buildings}. The black (resp. white) ones, when considered as vertices in the links of their end points, correspond to subspaces of dimension $1$ or $3$ (resp. $2$) in $\F_q^4$ and such an edge is contained in $(q^2+q+1)(q+1)$ (resp. $(q+1)^2$) pyramids. We denote by $\alpha^b$ the set of black edges of $\alpha$ and by $\alpha^w$ the set of white edges of $\alpha$.

The weight of a black (resp. white) edge is $w(e^b) = \frac{(q^2+q+1)(q+1)}{{4 \choose 2}|X(3)|}$ (resp. $w(e^w) = \frac{(q+1)^2}{{4 \choose 2}|X(3)|}$). Denote $\Theta = \frac{w(e^b)}{w(e^w)} = \frac{q^2+q+1}{q+1} \approx q$.
The norm of $\alpha$ is therefore
$$||\alpha|| = \frac{(q+1)^2}{{4 \choose 2}|X(3)|}(\Theta|\alpha^b| + |\alpha^w|).$$

It will be convenient in this section to use also the following norm of $\alpha$:

$$\uparrow \alpha \uparrow = \Theta |\alpha^b| + |\alpha^w|.$$

If $v$ is a vertex of $X$, then as before $\alpha_v$ is the set of edges of $\alpha$ with one endpoint in $v$, and $\alpha_v^b$ (resp. $\alpha_v^w$)
are the black (resp. white) ones. They give a $0$-cochain $\alpha_v \in C^0(X_v, \F_2)$ whose norm is

$$||\alpha_v || = \frac{(q^2+q+1)(q+1)}{{3 \choose 1}|X_v(2)|} |\alpha_v^b| + \frac{(q+1)^2}{{3 \choose 1}|X_v(2)|}|\alpha_v^w| = \frac{(q+1)^2}{{3 \choose 1}|X_v(2)|}(\Theta|\alpha_v^b| + |\alpha_v^w|).$$

Again, we denote

$$\uparrow\alpha_v\uparrow = \Theta |\alpha_v^b| + |\alpha_v^w|.$$

Note that $|X_v(2)|$ depends only on $q$, in fact,  $|X_v(2)| =(q^3+q^2+q+1)(q^2+q+1)(q+1) \approx q^6$.

Since $\alpha$ is locally minimal, for every vertex $v$, $\alpha_v$ is a minimal cochain of $C^0(X_v, \F_2)$, i.e., $\alpha_v$ is minimal in the coset $\alpha_v + B^0(X_v, \F_2)$, i.e., $||\alpha_v|| \leq ||\alpha_v + {\bf 1}_{X_v(0)}||$. This means that
$\Theta |\alpha_v^b| + |\alpha_v^w| \leq \frac{1}{2}(\Theta {\bf 1}^b_{X_v(0)} + {\bf 1}^w_{X_v(0)})$.

The righthand side is easily determined:

$${\bf 1}^b_{X_v(0)} = {4 \choose 1}_q + {4 \choose 3}_q = 2(q^3+q^2+q+1) \approx 2q^3.$$

$${\bf 1}^w_{X_v(0)} = {4 \choose 2}_q= \frac{(q^3+q^2+q+1)(q^2+q+1)}{(q+1)} \approx q^4.$$

Thus,

$\Theta |\alpha_v^b| + |\alpha_v^w| \leq \frac{1}{2}(\Theta 2 {4 \choose 1}_q  + {4 \choose 2}_q) \approx \frac{3}{2}q^4$.

Denote $Q = \Theta 2 {4 \choose 1}_q  + {4 \choose 2}_q \approx 3q^4$, so $\uparrow\alpha_v\uparrow \leq \frac{Q}{2}$.

%The norm of $\alpha$ is defined to be:
%$$||\alpha|| = \Theta|\alpha^b| + |\alpha^w|.$$
%$\alpha$ is locally minimal with regard to this norm, and $\alpha_v$ is defined also using this norm.
%We will denote $Q=\Theta(2{4 \choose 1}_q) + {4 \choose 2}_q \approx 3q^4$. $\alpha$ is locally minimal implies that for every $v$, $||\alpha_v|| \leq \frac{Q}{2}$.

\begin{lemma}~\label{lemma-traingles-counting-in-dim-3}
For $i=0,1,2,3$ denote by $t_i$, the number of triangles of $X$ which contain exactly $i$ edges from $\alpha$. Then,
\begin{enumerate}
\item~\label{item-one-lemma-traingles-counting-in-dim-3} $t_1+2t_2+3t_3=2(q+1)\uparrow\alpha\uparrow$.
\item~\label{item-two-lemma-traingles-counting-in-dim-3} $|\delta_1(\alpha)| = t_1+t_3$.
\item~\label{item-three-lemma-traingles-counting-in-dim-3} $\sum_{v \in X(0)}|E_{X_v}(\alpha_v, \overline{\alpha_v})|=2t_1+2t_2$.
\end{enumerate}
\end{lemma}

\begin{proof}
For (\ref{item-one-lemma-traingles-counting-in-dim-3}) we recall that the number of of triangles that contain a certain black edge is $2{3 \choose 1}_q$. The number of triangles that contain a certain white edge is $2(q+1)$. A triangle that contributes to $t_i$ contain $i$-edges from $\alpha$. Thus:
$$t_1+2t_2+3t_3 = 2 {3 \choose 1}_q |\alpha^b| + 2(q+1)|\alpha^w| = 2(q+1)(\frac{{3 \choose 1}_q}{q+1} |\alpha^b| + |\alpha^w|)=2(q+1)(\Theta|\alpha^b| + |\alpha^w|) = 2(q+1)\uparrow \alpha \uparrow$$

Part (\ref{item-two-lemma-traingles-counting-in-dim-3}) follows from the definitions.

For (\ref{item-three-lemma-traingles-counting-in-dim-3}): If $\triangle =\{v_0,v_1,v_2\}$ is a triangle of $X$, then it contributes an edge at $X_{v_k}$ ($\{ v_k\} = \{v_i,v_j,v_k\}\backslash \{v_i,v_j\} )$. This is the edge between $e_{i,k}=(v_i,v_k)$ and $e_{j,k}=(v_j,v_k)$ when we consider $e_{i,k}$ and $e_{j,k}$ as vertices of $X_{v_k}$. This edge will be in $E_{X_{v_k}}(\alpha_{v_k}, \overline{\alpha_{v_k}})$ if and only if exactly one of $\{e_{i,k},e_{j,k}\}$ is in $\alpha$. A case by case analysis of the four possibilities shows that if $\triangle$ has either $0$ or $3$ edges from $\alpha$ then $\triangle$ does not contribute anything to the left hand sum. On the other hand, if it has either $1$ or $2$ edges, it contributes $2$ to the sum. This proves the lemma.
\end{proof}

Fix now a small $\epsilon>0$ to be determined later and define:

\begin{definition} A vertex $v$ of $X$ is called {\em thin} w.r.t. $\alpha$ if $\uparrow \alpha_v \uparrow < (1-\epsilon)\frac{Q}{2}$ and {\em thick} otherwise (recall that by our local minimality assumption, $\uparrow \alpha_v \uparrow \leq \frac{Q}{2}$ for every $v$).

Denote
\begin{itemize}
\item $W=\{v \in V=X(0) | \mbox{ } \exists e \in \alpha \mbox{ with } v \in e\}$.
\item $R = \{v \in W | \mbox{ } v \mbox{ thin} \}$.
\item $S = \{v \in W | \mbox{ } v \mbox{ thick} \} = W \setminus R $.
\end{itemize}

Let $r=\sum_{v \in R}\uparrow \alpha_v \uparrow$ and $s=\sum_{v \in S} \uparrow \alpha_v \uparrow$.
\end{definition}

\begin{lemma}~\label{lemma-$r+s$-dim-3}
$r+s = 2\uparrow \alpha \uparrow$
\end{lemma}
\begin{proof} Every edge in $\alpha^b$ contributes $2\Theta$ to the left hand side and every edge of $\alpha^w$ contributes $2$ to the left hand side. So, $r+s = 2\Theta|\alpha^b| + 2|\alpha^w| = 2\uparrow \alpha \uparrow$.
\end{proof}

\begin{lemma}~\label{lemma-edges-exiting-alpha-v-dim-3}
\begin{enumerate}
\item~\label{item-one-lemma-edges-exiting-alpha-v-dim-3} For every $v \in V$, $|E_{X_v}(\alpha_v, \overline{\alpha_v})| \geq (q+1-\sqrt{12q})\uparrow \alpha_v \uparrow$.
\item~\label{item-two-lemma-edges-exiting-alpha-v-dim-3} If $v$ is thin, then $|E_{X_v}(\alpha_v, \overline{\alpha_v})| \geq (1+\epsilon)(q+1-\sqrt{12q})\uparrow \alpha_v \uparrow$.
\end{enumerate}
\end{lemma}

\begin{proof}
Recall (see Section~\ref{subsection-spherical-buildings} and the notations there) that the link graph $X_v$ is a $3$-partite graph with parts $M_1,M_2,M_3$, $\alpha_v = T_1 \cup T_2 \cup T_3$ where $T_i \subseteq M_i$.
We have,
$$|M_1| = |M_3| = {4 \choose 1}_q \approx q^3,$$
while
$$|M_2| = {4 \choose 2}_q \approx q^4.$$
Assume now $|T_i| = w_i|M_i|$.
In the graph  $Z_{1,2}$: $k' = q^2+q+1$, $k''=q+1$, the largest eigenvalue is $\sqrt{(q+1)(q^2+q+1)} \approx q^{\frac{3}{2}}$, the second largest eigenvalue is $\sqrt{q^2+q} \approx q+1$.
In the graph  $Z_{1,3}$: $k' = k'' = q^2+q+1 \approx q^2$, the largest eigenvalue is $q^2+q+1 \approx q^{2}$, the second largest eigenvalue is $q$.
%Thus we have the bipartite graphs $Z_{1,2}$, $Z_{3,2}$, $Z_{1,3}$ where $Z_{1,2} = Z_{3,2}$.
Using now Proposition~\ref{prop-mixing-for-bipartite} we have:
$$E(T_1,T_2) \leq \frac{q^{\frac{3}{2}} |T_1| |T_2|}{q^{\frac{7}{2}}} + (q+1)\sqrt{|T_1||T_2|} = \frac{1}{q^3}(q|T_1| \cdot |T_2|)+ (q+1)\sqrt{|T_1||T_2|}.$$
$$E(T_3,T_2) \leq \frac{q^{\frac{3}{2}} |T_3| |T_2|}{q^{\frac{7}{2}}} + (q+1)\sqrt{|T_3||T_2|}= \frac{1}{q^3}(q|T_3| \cdot |T_2|)+ (q+1)\sqrt{|T_3||T_2|}.$$
$$E(T_1,T_3) \leq \frac{q^{2} |T_1| |T_3|}{q^{3}} + q\sqrt{|T_1||T_3|}= \frac{1}{q^3}(q|T_1| \cdot q|T_3|)+ q\sqrt{|T_1||T_3|}.$$

Thus,

$|E_{X_v}(\alpha_v, \alpha_v)|\leq \frac{1}{q^3}(q|T_1| \cdot |T_2| + q|T_1| \cdot q|T_3| + q|T_3| \cdot |T_2|) + (q+1)(\sqrt{|T_1||T_2|} + \sqrt{|T_3||T_2|} + \sqrt{|T_1||T_3|})$.

Now using the Maclaurin's inequality $(yz+yw+zw) \leq \frac{1}{3}(y+z+w)^2$ we get

\begin{eqnarray}
|E_{X_v}(\alpha_v, \alpha_v)| & \leq & \frac{1}{q^3} \cdot \frac{1}{3}(q|T_1| + |T_2| + q|T_3|)^2 + (q+1)(\sqrt{|T_1||T_2|} + \sqrt{|T_3||T_2|} + \sqrt{|T_1||T_3|}) \\
&  \leq    & \frac{1}{q^3} \cdot \frac{1}{3}\uparrow \alpha_v \uparrow^2 + 3\sqrt{q}(\sqrt{q|T_1|\cdot|T_2| + q|T_3| \cdot|T_2| + q|T_1| \cdot q |T_3|}) \\
& \leq  & \frac{1}{q^3} \cdot \frac{1}{3}\uparrow \alpha_v \uparrow^2 + \sqrt{3q} \uparrow \alpha_v \uparrow.
\end{eqnarray}

Since the degree of a vertex in $M_2$ is $2(q+1)$ while for a vertex in $M_1 \cup M_3$ it is $2(q+1)\Theta$, we obtain

$|E_{X_v}(\alpha_v, \overline{\alpha_v})| \geq 2(q+1)\uparrow \alpha_v \uparrow -\frac{1}{q^3} \cdot \frac{2}{3}\uparrow \alpha_v\uparrow^2 - 2\sqrt{3q}\uparrow \alpha_v \uparrow = (2(q+1) -\frac{1}{q^3} \cdot \frac{2}{3}\uparrow \alpha_v \uparrow - \sqrt{12q})\uparrow \alpha_v \uparrow$.

Moreover, since $\uparrow \alpha_v \uparrow \leq \frac{Q}{2} \approx \frac{3}{2}q^4$ we get:

$|E_{X_v}(\alpha_v, \overline{\alpha_v})| \geq (2(q+1) -\frac{1}{q^3} \cdot \frac{2}{3}\cdot\frac{3}{2}q^4 - \sqrt{12q})\uparrow \alpha_v \uparrow
\geq (q+1 -\sqrt{12q})\uparrow \alpha_v \uparrow$,

and Part~\ref{item-one-lemma-edges-exiting-alpha-v-dim-3} of the lemma is proved.

Now, if $v$ is thin then $\uparrow \alpha_v \uparrow \leq (1-\epsilon) \cdot \frac{Q}{2} = (1-\epsilon) \cdot \frac{3}{2}q^4$; so in this case:

$|E_{X_v}(\alpha_v, \overline{\alpha_v})| \geq (2(q+1) -\frac{1}{q^3} \cdot \frac{2}{3} \cdot (1-\epsilon) \cdot \frac{3}{2}q^4 - \sqrt{12q})\uparrow \alpha_v \uparrow \geq (1+\epsilon)(q+1 -\sqrt{12q})\uparrow \alpha_v \uparrow$,

and Part~\ref{item-two-lemma-edges-exiting-alpha-v-dim-3} of the lemma is also proved.
\end{proof}

We can deduce the following inequality

\begin{lemma}~\label{lemma-2t-1+2t-2-dim-3}
$2t_1 + 2t_2 \geq 2(q+1-\sqrt{12q})\uparrow \alpha \uparrow + \epsilon (q+1-\sqrt{12q}) r$.
\end{lemma}

\begin{proof}
\begin{eqnarray}
2t_1 + 2t_2 = \sum_{v \in W}E_{X_v}(\alpha_v, \overline{\alpha_v})
& =    &  \sum_{v \in R}E_{X_v}(\alpha_v, \overline{\alpha_v}) + \sum_{v \in S}E_{X_v}(\alpha_v, \overline{\alpha_v}) \\
& \geq &  \sum_{v \in R}(1+\epsilon)(q+1-\sqrt{12q})\uparrow \alpha_v \uparrow + \sum_{v \in S}(q+1-\sqrt{12q})\uparrow \alpha_v \uparrow \\
& =    &  (1+\epsilon)(q+1-\sqrt{12q})r + (q+1-\sqrt{12q})s \\
& =    &  (q+1-\sqrt{12q})(r+s)+\epsilon (q+1-\sqrt{12q})r \\
& =    &2(q+1-\sqrt{12q})\uparrow \alpha \uparrow + \epsilon(q+1-\sqrt{12q})r
\end{eqnarray}
In the first equation we have used Lemma~\ref{lemma-traingles-counting-in-dim-3}, part (\ref{item-three-lemma-traingles-counting-in-dim-3}) and in the last one Lemma~\ref{lemma-$r+s$-dim-3}. The inequality follows from Lemma~\ref{lemma-edges-exiting-alpha-v-dim-3}.
\end{proof}
%\tnote{check, should it be epsilon divided by two?}

\begin{lemma}\label{lemma-t-1-minus-3t-3-dim-3}
$t_1 - 3t_3 \geq \epsilon (q+1-\sqrt{12q})r - 2\sqrt{12q}\uparrow\alpha\uparrow$.
\end{lemma}

\begin{proof} Subtract equation (\ref{item-one-lemma-traingles-counting-in-dim-3}) in Lemma~\ref{lemma-traingles-counting-in-dim-3} form the equation obtained in Lemma~\ref{lemma-2t-1+2t-2-dim-3}.
\end{proof}

Our next goal now is to show the existence of $\eta_1 > 0$, such that for every $\alpha$ with $||\alpha|| \leq \eta_1$, the contribution $r$ of the thin edges, is at least some fixed fraction of $\uparrow \alpha \uparrow$. This will prove that for $q$ large enough $t_1 \geq cq\uparrow\alpha\uparrow$, and the case $i=1$ of Theorem~\ref{thm-isoperimetric -inequalities} will follow with a constant $\epsilon_1$, which is independent of $q$ (for all $q \gg 0$). Indeed, $$||\delta_1(\alpha)|| = \frac{(q+1)(t_1(\alpha)+t_3(\alpha))}{{4 \choose 3}|X(3)|}\geq \frac{q \cdot t_1(\alpha)}{4 |X(3)|} \geq \frac{c q^2 \uparrow\alpha\uparrow}{4 |X(3)|} = \frac{1}{4|x(3)|}\cdot \frac{{4 \choose 2} |X(3)|}{(q+1)^2} \cdot c q^2 ||\alpha|| \geq \epsilon_1||\alpha||,$$

for a suitable constant $\epsilon_1 \geq 0$

Up to now we have used only the local structure of $X$. Now we will use the global structure; the fact that its $1$-skeleton is an almost Ramanujan graph.

Recall, that by Corollary~\ref{cor-bounding-lambda-in-quotient-Ram-complex}, the second largest eigenvalue of the adjacency matrix of $X^{(1)}$ is bounded from above by $(d+1)^{d+1}\sqrt{k} \approx 4^4\sqrt{q^4} = 4^4q^2$. Now the degree of a vertex is $k \approx q^4$ so $\lambda_1(X^{(1)}) \geq q^4 - 4^4q^2$. Note now that for every vertex in $S$ we have $\uparrow \alpha_v\uparrow \geq (1-\epsilon)\frac{Q}{2}$.
So, every $v \in S$ either touches at least $\frac{(1-\epsilon)q^4}{2}$ white edges or at least $\frac{(1-\epsilon)q^3}{2}$ black edges.
Let us denote by $S^b$ the vertices in which the first case occurs and by $S^w$ the vertices in which the second case occurs (it could be that both cases occur at $v$). Then $|S^b| \leq \frac{2|\alpha^b|}{(1-\epsilon)q^3/2}$ and $|S^w| \leq \frac{2|\alpha^w|}{(1-\epsilon)q^4/2}$.

Thus,
$|S^b \cup S^w| \leq |S^b| + |S^w| \leq \frac{2q|\alpha^b|}{(1-\epsilon)q \cdot q^3/2} + \frac{2|\alpha^w|}{(1-\epsilon)q^4/2} = \frac{2}{(1-\epsilon)q^4/2}(q|\alpha^b| + |\alpha^w|) \leq  \frac{4}{(1-\epsilon)q^4}\uparrow\alpha\uparrow$.

Hence $|S| \leq \frac{4\uparrow\alpha\uparrow}{(1-\epsilon)q^4}$. Let $n: = |V| = |X(0)|$, we have,

\begin{eqnarray}
|E(S)| & \leq & \frac{1}{2}(q^4-\frac{|\overline{S}|}{n}\lambda_1(X^{(1)}))|S| \\
     & \leq & \frac{1}{2}(q^4-\frac{|\overline{S}|}{n}(q^4-4^4q^2))|S| =  \frac{1}{2}(q^4(1-\frac{|\overline{S}|}{n})+4^4q^2 \frac{|\overline{S}|}{n})|S|\\
     & \leq & \frac{1}{2}(q^4 \frac{|S|}{n}+4^4q^2)|S| \leq \frac{1}{2}(\frac{4\uparrow\alpha\uparrow}{(1-\epsilon)n}+4^4q^2)|S|
\end{eqnarray}

Note that we assumed that $||\alpha|| \leq \eta_1$. As $||\alpha|| = \frac{(q+1)^2}{{4 \choose 2}|X(3)|} \uparrow\alpha\uparrow$ and $|X(3)| \approx \frac{n \cdot q^6}{4}$, we have,

$$ \uparrow\alpha\uparrow \leq \frac{3}{2}\eta_1q^4 n. $$

Hence,

\begin{eqnarray}
|E(S)| & \leq & \frac{1}{2} (\frac{4 \frac{3}{2}\eta_1q^4 n}{(1-\epsilon)n} + 4^4q^2)|S| \\
       & \leq & (\frac{3\eta_1q^4}{(1-\epsilon)} + 128q^2) \frac{4\uparrow\alpha\uparrow}{(1-\epsilon)q^4} \\
       & \leq & (\frac{12\eta_1}{(1-\epsilon)^2} + \frac{512}{(1-\epsilon)q^2}) \uparrow\alpha\uparrow
\end{eqnarray}

Thus, $|E(S)| \leq (\frac{12\eta_1}{(1-\epsilon)^2} + \frac{512}{(1-\epsilon)q^2}) (\Theta |\alpha^b| + |\alpha^w|)$. Thus, for $q \gg 0$, only a small fraction (less than $\mu =  \frac{12}{(1-\epsilon)q^{0.1}}$) of the black edges are between thick vertices and even a smaller fraction of the white ones. Namely, all the rest have at least one thin end vertex. This implies that $r \geq (1-\mu) \Theta |\alpha^b| + (1-\mu)  |\alpha^w| = (1-\mu)\uparrow\alpha\uparrow$. Theorem~\ref{thm-isoperimetric -inequalities} is now proved also for $i=1$.

%The proof is effective. One can estimate $\epsilon_1$ and how large should be $q$, in term of $\epsilon_0$. Furthermore, we used the global structure only in the last argument. It would work (with a price of having $q$ larger) also if we use only the weaker bound on $\lambda_1(X^{(1)})$ obtained from property $(T)$ (see Proposition~\ref{prop-propery-T-esimate-on-lambda-1}) rather than the strong bound from Corollary~\ref{cor-bounding-lambda-in-quotient-Ram-complex}. So, Theorem~\ref{thm-two-dim-coboundary-exp} is valid for all the finite quotients of $B=\tilde{A}_2(F)$ provided $q \gg 0$.
%\tnote{check!}

Let us mention that along the way we have proved two facts which are worth formulating separately.

\begin{cor}\label{cor-from-3-dim-result} In the notations and assumptions as above. If $q \gg 0$, then we have:
%in particular, $q \gg 0$: There exists $\rho > 0$ such that:
\begin{enumerate}
\item If $\alpha \in B^1(X,\F_2)$ is a locally minimal coboundary with $||\alpha|| < \frac{1}{q^{1.1}}$ then $\alpha = 0$.
\item If $\alpha \in Z^1(X,\F_2) \setminus B^1(X,\F_2)$ then $||\alpha|| > \eta_1$. In particular, for a fixed $q$, every representative of a non-trivial cohomology class has linear size support.
\end{enumerate}
\end{cor}

\begin{proof} The first item follows immediately since $\delta(\alpha)=0$. For the second item we observe that every $\alpha \in Z^1$, can be replaced by a locally minimal representative $\alpha'$ with $||\alpha'|| \leq ||\alpha||$ and $\alpha'  = \alpha (\mbox{mod }B^1)$. Applying now Theorem~\ref{thm-isoperimetric -inequalities} for $\alpha'$, we deduce the result.
\end{proof}

\begin{remark} In our proof of the Theorem $\epsilon_1$ turns out to be independent of $q$ (provided $q \gg 0$), but $\eta_1$ does depend on $q$ (we choose $\eta_1 \approx \frac{1}{q^{1.1}}$). One can improve the proof to make also $\eta_1$ independent of $q$ by considering the "black skeleton" of $X$, i.e., the subgraph of $X^{(1)}$ consisting of the black edges and only them.
%\tnote{is this true?}
\end{remark}

Note that by Proposition~\ref{prop-non-trivial-first-cohomology}, there are indeed cases with $H^1(X,\F_2) \neq \{0\}$, so the second item of Corollary~\ref{cor-from-3-dim-result} is a non-vacuous systolic result. Such results are of potential interest for quantum error-correcting codes (see \cite{LubotzkyGuth},\cite{Zemor} and the references therein).

We move now to the third case, i.e., $i=2$, in which we have to prove $2$-expansion. This case is by far more difficult (and we have to overcome along the way the difficulties of the case $i=1$, but also much more.) This will be the topic of the next section.

\section{Expansion of $2$-cochains in $3$-dimensional Ramanujan complexes}
\label{section-three-dim-coboundary-exp}

In this section we prove the case $i=2$ of Theorem~\ref{thm-isoperimetric -inequalities}. We prove that for $q\gg 0$, exists $\epsilon''>0$ such that if $\alpha \in C^2(X,\F_2)$ is locally minimal with $|\alpha | < q^3 |X(0)|$ then $|\delta (\alpha)| \geq \epsilon'' q|\alpha|$. This indeed will prove the theorem: recall that every $2$-cell is contained in $q+1$ pyramides and altogether there are approximately $q^6|X(0)|$ pyramides.
Thus,
$$||\alpha|| \approx \frac{(q+1)|\alpha|}{{4 \choose 3} q^6 |X(0)|} \approx c' \frac{|\alpha|}{q^5 |X(0)|},$$
and  $||\delta(\alpha)|| = \frac{|\delta(\alpha)|}{q^6 |X(0)|}$. Hence, we can deduce the result with $\mu_3 = \frac{c''}{q^2}$ and $\epsilon_3$ independent of $q$ (provided $q \gg 0$).

First recall again that the link $X_v$ of every vertex $v$ of $X$ is the $2$-dimensional spherical building $S(4,q)$. The vertices of $X_v$ are of two types: the one corresponding to subspaces of dimensions $1$ and $3$ of $\F_q^4$, the {\em black vertices}, and the other type corresponding to subspaces of dimension $2$, which we call the {\em white vertices}. A black vertex is of degree $2(q^2+q+1)$, while a white one is of degree $2(q+1)$. There are approximately $q^3$ black vertices and approximately $q^4$ white ones.

Given a vertex $v$, an edge $e$ of $X$, with $v \in e$, gives a unique vertex in $X_v$, which can be black or white, and we then call $e$ black or white, accordingly.

%As $|X(1)| \approx q^4 |X(0)|$ and $|X(2)| \approx q^5 |X(0)|$, once we prove that $\epsilon''$ as above exists, it proves that $\epsilon_2$ in Theorem~\ref{thm-isoperimetric -inequalities} is independent of $q$ (for $q \gg 0$), but the fact that we are proving the result only for $\alpha$, with $|\alpha| < q^3|X(0)|$ (and not for all $\alpha$ with $|\alpha| < q^5|X(0)|$) means that $\eta_2$ does depend on $q$ (in fact, $\eta_2 = \frac{1}{q^2}$). This is weaker that what we proved in Theorem~\ref{thm-two-dim-coboundary-exp}, and for the cases of $i=0,1$ of Theorem~\ref{thm-isoperimetric -inequalities}, where $\eta_0,\eta_1,\epsilon_0$ and $\epsilon_1$, are independent of $q$ (for $q \gg 0$). For our applications here the dependence on $q$ is not crucial, but it will be interesting to prove the strongest statement.

%Anyway, let us present the proof:

%\tnote{add in previous section that in Theorem~\ref{thm-two-dim-coboundary-exp} the constants do NOT depend on $q$}

We start with local considerations. Since $\alpha \in C^2(X,\F_2)$ is locally minimal, for every edge $e$ of $X$, $|\alpha_e| \leq \frac{|X_e(0)|}{2}$. Indeed, $\alpha$ being locally minimal means that for every vertex $v$ of $X$, $\alpha_v \in C^1(X_v, \F_2)$ is minimal and in particular, it is locally minimal as a cochain of $X_v$. Thus, at every vertex $w$ of $X_v$ $\alpha_v$ contains at most half of the edges around $w$. This exactly means that the number of triangles in $\alpha$ containing $e$ is at most half of all the triangles containing $e$.

For $i=0,\cdots,4$ we denote by $t_i$ the number of pyramides ($3$-cells) in $X$ which contain exactly $i$ triangles of $\alpha$.

\begin{lemma}~\label{lemma-t-1+2t-2+3t-3+4t-4}
\begin{enumerate}
\item $\sum_{e \in X(1)}|\alpha_e| = 3|\alpha|$.
\item~\label{item-two-in-lemma-t-1+2t-2+3t-3+4t-4} $t_1+2t_2+3t_3+4t_4=(q+1)|\alpha|$.
\end{enumerate}
\end{lemma}

\begin{proof} The first item follows from the fact that every triangle has three edges. The second is because every triangle is contained in exactly $q+1$ pyramides.
\end{proof}

\begin{lemma}\label{lemma-3t-1+4t-2+3t-3} $\sum_{e \in X(1)} E_{X_e}(\alpha_e,\overline{\alpha_e})=3 t_1 +4 t_2 +3 t_3$.
\end{lemma}

\begin{proof}
Recall that $\alpha_e$ can be considered as a set of vertices in the $(q+1)$-regular graph $X_e$, which is the link of $X$ at $e$. If $P$ is a pyramid with one triangle from $\alpha$, then for $3$ out of the $6$ edges of $P$, $|\alpha_e|=1$  and for $3$ of them $|\alpha_e|=0$. The first $3$ contributes, each, $1$ to the left hand side and the later have no contribution. If $P$ has $2$ triangles from $\alpha$, then for one edge $|\alpha_e| =2$ but this edge contributes nothing to the LHS, since $P$ represents then an edge of $X_e$ from $\alpha_e$ to $\alpha_e$. For $4$ other edges of $P$, $|\alpha_e|=1$ and each contributes $1$ to the LHS. For the last edge, $|\alpha_e|=0$ and clearly no contribution to the LHS. A similar consideration justifies the claim about the $3t_3$ contribution (or by duality to $t_1$). Pyramids with either $0$ or $4$ triangles from $\alpha$ contributes nothing to the LHS.
\end{proof}

\begin{definition}(Thin/thick edge)~\label{def-thin-thick-edge}
An edge $e \in X(1)$ is called {\em thin} if $|\alpha_e| \leq |X_e(0)|^{0.9}$ and {\em thick} otherwise.

Denote:

$R$ (resp. $S$) - the set of thin (resp. thick) edges.

$r := \sum_{e \in R}|\alpha_e|$.

$s := \sum_{e \in S}|\alpha_e|$.

So, by Lemma~\ref{lemma-t-1+2t-2+3t-3+4t-4}, $r+s = 3|\alpha|$.
\end{definition}

The link graph $X_e$ of every edge $e$ of $X$ is either the "points versus lines graph" of the projective plane $\P_2(q)$ over $\F_q$ or the complete $(q+1)$-bipartite graph on $2(q+1)$ vertices. Indeed, if $e$ is a white edge, $X_e$ is the complete bipartite ($q+1$)-regular graph on $2(q+1)$ vertices. While if $e$ is black, $X_e$ is the "points versus lines" graph of the projective plane of $\F_q$, i.e., ($q+1$)-regular bipartite on $2(q^2+q+1)$ vertices. When $e$ is a thick/thin edge of $X$, we will also consider it as thick/thin vertex of $X_v$, for $v \in e$.
In either case, just as with the $2$-dimensional complexes studied in Section~\ref{section-two-dim-proof}, $\lambda_1(X_e) \geq q+1-\sqrt{q}$.
The next lemma follows from Proposition~\ref{prop-cheeger}.

\begin{lemma}~\label{lemma-edges-out-of-thin-edge}
\begin{enumerate}
\item For every $e \in X(1)$, $E_{X_e} (\alpha_e, \overline{\alpha_e}) \geq \frac{1}{2}(q+1 - \sqrt{q})|\alpha_e|$.
\item If $e$ is thin then $E_{X_e} (\alpha_e, \overline{\alpha_e}) \geq (q+1-q^{0.9})|\alpha_e|$.
\end{enumerate}
\end{lemma}

\begin{proof} The first item is deduced directly from Proposition~\ref{prop-cheeger}. For the second item, assume first that $e$ is white. In this case Proposition~\ref{prop-cheeger} gives:
$$E_{X_e} (\alpha_e, \overline{\alpha_e}) \geq \frac{2(q+1)-(2(q+1))^{0.9}}{2(q+1)}(q+1 - \sqrt{q})|\alpha_e| = (1 - \frac{1}{(2(q+1))^{0.1}})(q+1 - \sqrt{q})|\alpha_e| \geq (q+1 - q^{0.9})|\alpha_e|,$$
as $q$ is large. Similarly this is also true for black edges.
\end{proof}

Combining Lemmas~\ref{lemma-3t-1+4t-2+3t-3} and~\ref{lemma-edges-out-of-thin-edge} we get:

\begin{lemma}\label{lemma-3t-1+4t-2+3t_3-large-if-many-thins} $3t_1 + 4t_2 +3t_3 \geq \frac{3}{2}(q+1)|\alpha| + \frac{(q+1)}{2}r -3q^{0.9}|\alpha|$.
\end{lemma}

\begin{proof}
\begin{eqnarray}
3t_1 + 4t_2 +3t_3 & =    & \sum_{e \in S}E_{X_e}(\alpha_e, \overline{\alpha_e}) + \sum_{e \in R}E_{X_e}(\alpha_e, \overline{\alpha_e}) \\
                  & \geq & \frac{1}{2}(q+1-\sqrt{q})s + (q+1-q^{0.9})r \\
                  & \geq & \frac{3}{2}(q+1)|\alpha| + \frac{(q+1)}{2}r -3q^{0.9}|\alpha|.
\end{eqnarray}
\end{proof}

Subtracting twice Lemma~\ref{lemma-t-1+2t-2+3t-3+4t-4}(\ref{item-two-in-lemma-t-1+2t-2+3t-3+4t-4}) from the equation proved in Lemma~\ref{lemma-3t-1+4t-2+3t_3-large-if-many-thins} we get:

\begin{lemma}\label{lemma-t-1-large-if-many-thins} $t_1 - 3t_3  -8t_4 \geq -\frac{(q+1)}{2}|\alpha| + \frac{(q+1)}{2}r -3q^{0.9}|\alpha|$.
\end{lemma}

Thus, in order to prove Theorem~\ref{thm-isoperimetric -inequalities}, it will suffice to prove that $r > (1+\epsilon')|\alpha|$  for some $\epsilon'$ (independent of $q$) when $q$ is large. I.e., more than $\frac{1}{3}$ of the contribution to $\alpha$ comes from thin edges. It is interesting to compare this with the proof in Section~\ref{section-two-dim-proof} for $\mbox{dim }X=2$, where we had only to show that $r \geq \epsilon|\alpha|$. This is what makes the current proof more difficult.

Let us now use a global argument.

\begin{definition}(Thin/thick vertex)~\label{def-thin-thick-vertex}
A vertex $v \in X(0)$ is called a {\em thin vertex} if in its link, $X_v$, there are less than $q^{2.75}$ thick black vertices and less than $q^{3.7}$ thick white vertices. Otherwise it is called a {\em thick vertex}. Let $S_0$ denote the set of thick vertices, and $R_0$ - the thin ones.
\end{definition}

For every $v \in X(0)$, our cochain $\alpha$ defines a $1$-cochain $\alpha_v \in C^1(X_v,\F_2)$.

\begin{lemma}~\label{lemma-thick-vertex}
\begin{enumerate}
\item~\label{item-one-lemma-thick-vertex}  If $v$ is a thick vertex then $|\alpha_v| \geq \frac{q^{4.55}}{2}.$
\item~\label{item-two-lemma-thick-vertex} The number of thick vertices is at most $\frac{6|\alpha|}{q^{4.55}}$, i.e., $|S_0| \leq \frac{6|\alpha|}{q^{4.55}}$.
\item~\label{item-three-lemma-thick-vertex} $|S_0| \leq \frac{6n}{q^{1.55}}$, where $n := |X(0)|$.
\end{enumerate}
\end{lemma}

\begin{proof}

(\ref{item-one-lemma-thick-vertex}) If $v$ is a thick vertex than it lies on either at least $q^{2.75}$ black thick edges or on at least $q^{3.7}$ white thick edges. A black (resp. white) thick edge is contained in at least $(2(q^2+q+1))^{0.9}$ (resp. $(2(q+1))^{0.9}$) triangles.
The same triangle can be counted twice according to its two edges which touch $v$, but in any case it means that there are at least $\frac{q^{4.55}}{2}$ edges in $\alpha_v$.

(\ref{item-two-lemma-thick-vertex}) By (\ref{item-one-lemma-thick-vertex}), every thick vertex touches at least $\frac{q^{4.55}}{2}$ triangles from $\alpha$.
A triangle touches $3$ vertices, so it can be counted at most $3$ times. Hence $|S_0| \leq \frac{6|\alpha|}{q^{4.55}}$.

(\ref{item-three-lemma-thick-vertex}) follows from the fact that $|\alpha| \leq q^3n$.
\end{proof}

%Recall now our ongoing assumption in Theorem~\ref{thm-isoperimetric -inequalities} (where we take $\rho=1$) which says that $|| \alpha || \leq q^3n$, where $n =|X(0)|$. Thus, we have $|S_0| < \frac{3q^3n}{q^{4.55}} = \frac{3}{q^{1.55}} \cdot n$.

As $X^{(1)}$ is "almost" a Ramanujan graph we can prove:

\begin{lemma}~\label{lemma-number-of-edges-within-S-0}
For $q \gg 0$, $|E(S_0, S_0)| \leq \frac{19}{q^{2.1}}|\alpha|$.
\end{lemma}

\begin{proof}
The graph $X^{(1)}$ is a $k$-regular graph with $k=\sum_{i=1}^{3}{4 \choose i}_q \approx q^4$ and $\lambda_1(X^{(1)}) \geq k-4^4\sqrt{k}$
(see Corollary~\ref{cor-bounding-lambda-in-quotient-Ram-complex}). Thus, by Proposition~\ref{prop-cheeger}, $|E(S_0,\overline{S_0})| \geq \frac{\overline{S_0}}{n}(k-4^4\sqrt{k})|S_0|$. Hence,
\begin{eqnarray}
|E(S_0,S_0)| & =    & \frac{1}{2}(k|S_0| - |E(S_0,\overline{S_0})|) \\
           & \leq & \frac{1}{2}(k|S_0| - \frac{|\overline{S_0}|}{n}(k-4^4\sqrt{k})|S_0| )\\
           & =    & \frac{|S_0|}{2}(k(1-\frac{|\overline{S_0}|}{n}) + \frac{|\overline{S_0}|}{n} 4^4\sqrt{k}) \\
           & =    & \frac{|S_0|}{2}(k\frac{|S_0|}{n} + \frac{|\overline{S_0}|}{n} 4^4\sqrt{k}) \\
           & \leq & \frac{|S_0|}{2}(k\frac{|S_0|}{n} +  4^4\sqrt{k})
%            \frac{1}{2}(k\frac{|S_0|}{n} + 4\sqrt{k})|S_0| \leq \frac{1}{2}(\frac{3}{q^{1.55}q^4+4q^2})\frac{3||\alpha||}{q^{4.55}} \leq \frac{5}{q^{2.1}}||\alpha||.$$
\end{eqnarray}
Now by Lemma~\ref{lemma-thick-vertex},
$|E(S_0,S_0)| \leq \frac{3|\alpha|}{q^{4.55}}(q^4 \frac{6}{q^{1.55}} + 4^4q^2) \leq \frac{19|\alpha|}{q^{2.1}}$.
\end{proof}

Thus, for $q$ large enough, only small proportion of the triangles in $\alpha$ have two (or more) thick vertices. Indeed,  the total number of edges between thick vertices is bounded by $\frac{20}{q^{2.1}}|\alpha|$ and on every edge there are at most $q^2+q+1$ triangles from $\alpha$. So we have the following corollary, from which we conclude that almost every triangle of $\alpha$ has at most one thick vertex.

\begin{cor}\label{cor-almost-every-alpha-triangle-has-at-most-one-thick-vertex}
There are at most $\frac{20}{q^{0.1}}|\alpha|$ triangles with $2$ or $3$ thick vertices.
\end{cor}

Now we show that almost all the triangles in $\alpha$ have at least one thin edge.

\begin{lemma}~\label{lemma-neighborhood-of-thin-vertex}
\begin{enumerate}
\item~\label{item-one-lemma-neighborhood-of-thin-vertex}  The number of triangles of $\alpha$ with at least one thin edge is at least $(1-o_q(1))|\alpha|$.

\item~\label{item-two-lemma-neighborhood-of-thin-vertex} The fraction of triangles of $\alpha$ with 3 thin vertices and at most one thin edge is $o_q(1)$.
%At least $(1-o_q(1))$ fraction of the triangles of $\alpha$ with $3$ thin vertices, have at least $2$ thin edges.
\end{enumerate}
\end{lemma}

Both parts of the lemma follow from the following result:

\begin{lemma}~\label{lemma-neighborhood-of-thin-vertex-basic}
Let $v$ be a thin vertex of $X$, $X_v$ its link, and $\alpha_v$ the $1$-cochain of $X_v$ induced by $\alpha$.
Let $S_v$ be the subset of $X_v(0)$ of the thick vertices of $X_v$, i.e., the ones corresponding to the thick edges of $X$ coming out of $v$.
Then:
$$\frac{|E(S_v,S_v)|}{|\alpha_v|} = o_q(1). $$
Namely, the probability of an edge of $\alpha_v$ on $X_v^{(1)}$ coming from a vertex in $S_v$ to stay at $S_v$ is going to zero as $q \rightarrow \infty$.
\end{lemma}

\begin{proof} The link $X_v$ is isomorphic to $S(4,q)$, so $S_v$ can be considered as a subset of $M=M_1 \cup M_2 \cup M_3$ of the subspaces of $\F_q^4$, where $M_i$ is the set of subspaces of dimension $i$. Let $T_i = S_v \cap M_i$. As $v$ is a thin vertex of $X$, Definition~\ref{def-thin-thick-vertex} implies that $|T_2| < q^{3.7}$ and $|T_1 \cup T_3| < q^{2.75}$. On the other hand, every $w \in S_v$ corresponds to a thick edge $e$ of $X$. Hence, by Definition~\ref{def-thin-thick-edge}, $e$ lies on at least $|X_e(0)|^{0.9}$ triangles from $\alpha$, or in other words $\alpha_v$ has at least $(\mbox{deg }(w))^{0.9}$ edges coming out of $w$. Now, if $w \in T_1 \cup T_3$, $\mbox{deg}(w) = 2(q+1)$, while if $w \in T_2$, $\mbox{deg}(w) = 2(q^2+q+1)$, thus all the assumptions of Lemma~\ref{lemma-neighborhood-of-thin-vertex-defs} are satisfied and our lemma follows.
\end{proof}
%\begin{lemma}~\label{lemma-neighborhood-of-thin-vertex-basic}
%Let $v$ be a thin vertex of $X$, $X_v^{(1)}$ the $1$-skeleton of its link, and $S_v$ the set of thick vertices in $X_v(0)$ (i.e., the ones corresponding to the thick edges of $X$ containing $v$). Let $Y$ be the subgraph of $X_v(1)$ induced by $\alpha_v$ (i.e, the subgraph whose edges comes from the triangles in $\alpha$). For every $w \in S_v$, denote by $\mbox{deg}_y(w)$ the degree of $w$ in the graph $Y$ (note that if $w$ is black then $\mbox{deg}_y(w) \geq (2(q^2+q+1))^{0.9}$ and if it is white then $\mbox{deg}_y(w) \geq (2(q+1))^{0.9}$). Let $e(S_v,Y)=\sum_{w \in S_v}\mbox{deg}_y(w)$, then $$\frac{E(S_v, \overline{S_v})}{e(S_v,Y)} = 1+o_q(1) .$$
%Namely, the neighbors of a thick edge at $v$ is almost surely (when $q \gg 0$) a thin edge, or better, if $v$ is a thin vertex with a thick edge on a triangle from $\alpha$, then almost surely the other edge of the triangle will be thin
%\end{lemma}
%
%\begin{proof}
%The Lemma follows directly from Lemma~\ref{lemma-neighborhood-of-thin-vertex-defs}.
%\end{proof}

Let us spell out the meaning of the last lemma. Lemma~\ref{lemma-neighborhood-of-thin-vertex-basic} says that if $e$ is a thick edge of $X$ containing a thin vertex $v$, and $\triangle$ is a triangle in $\alpha$ containing $e$, then the second edge of $\triangle$ touching $v$, is most likely thin (with probability $1-o_q(1)$).

We next show that Lemma~\ref{lemma-neighborhood-of-thin-vertex} follows from Lemma~\ref{lemma-neighborhood-of-thin-vertex-basic}.

\begin{proof} (of Lemma~\ref{lemma-neighborhood-of-thin-vertex})
As we saw above in Corollary~\ref{cor-almost-every-alpha-triangle-has-at-most-one-thick-vertex}, almost all the triangles of $\alpha$ have at least one thin vertex and hence, by Lemma~\ref{lemma-neighborhood-of-thin-vertex-basic} and the remark afterwards, at least one of the edges of the triangle adjacent to it is thin (with probability $1-o_q(1)$). This proves Lemma~\ref{lemma-neighborhood-of-thin-vertex}(\ref{item-one-lemma-neighborhood-of-thin-vertex}).

Similarly, consider the triangles of $\alpha$ with three thin vertices and at most one thin edge, namely, those with three thin vertices and at least two thick edges. Each such a triangle $\triangle$ contributes 1 to $|E(S_v,S_v)|$ for a vertex $v$ that is between its two thick edges (using the
notations of Lemma~\ref{lemma-neighborhood-of-thin-vertex}). Thus the total number $T$ of such triangles is bounded by $\sum_{v \in R_0}|E(S_v,S_v)|$, where $R_0$ is the set of thin vertices.
By Lemma~\ref{lemma-neighborhood-of-thin-vertex}, we have that $\frac{|E(S_v,S_v)|}{|\alpha_v|} = o_q(1). $ Thus,
$T \leq \sum_{v \in R_0}|E(S_v,S_v)| \leq \sum_{v \in R_0} o_q(1)|\alpha_v| \leq  o_q(1)\cdot 3|\alpha|$. I.e., the fraction of triangles of $\alpha$ with 3 thin vertices and at least two thick edges is $o_q(1)$. This proves Lemma~\ref{lemma-neighborhood-of-thin-vertex}(\ref{item-two-lemma-neighborhood-of-thin-vertex}).

 %consider the triangles of $\alpha$ such that all their vertices are thin. Those that have only thin edges in them support our claim, so the rest have a thick edge such that two of its end points are thin vertices. Such triangles (with a thick edge such that two of its end points are thin vertices) contain by Lemma~\ref{lemma-neighborhood-of-thin-vertex-basic} another two thin edges with probability $1-o_q(1)$. This proves Lemma~\ref{lemma-neighborhood-of-thin-vertex}(\ref{item-two-lemma-neighborhood-of-thin-vertex}).
\end{proof}

Let now state a general observation about any cochain $\beta \in C^2(X,\F_2)$. Such $\beta$ induces two cochains on the link $X_v$ of every vertex $v$: One is $\beta_v^1 = \beta_v \in C^1(X,\F_2)$ that we have used so far, and the other is $\beta_v^2 \in C^2(X_v,\F_2)$ which is defined just by restricting $\beta$ to the triangles of $X_v$, when we recall that the link of $v$ is the set of simplicies $\tau $ of $X$ s.t. $v \notin \tau$ and $\tau \cup \{v\}$ is also a simplex of $X$. The next proposition follows from the definitions.

\begin{proposition}\label{prop-b-v-1-and-b-v-2} If $\beta \in C^2(X,\F_2)$ then
$$|\delta_2 \beta|= \frac{1}{4}\sum_{v \in X(0)}|\delta_1 \beta_v^1 + \beta_v^2 | \geq
\frac{1}{4}(\sum_{v \in X(0)}|\delta_1 \beta_v^1 | -   \sum_{v \in X(0)}|\beta_v^2 | )
= \frac{1}{4}(\sum_{v \in X(0)}|\delta_1 \beta_v^1 |) -   \frac{(q+1)}{4}|\beta |.$$
\end{proposition}

\begin{proof}
Recall that $\delta_2 \beta$ is the set of $3$-cells of $X$ which contains an odd number of triangles from $\beta$. The first equality is just summing up over the vertices $v$ of $X$, the number of such $3$-cells of $\delta_2 \beta$ touching $v$ is indeed the same as the number of triangles in $\delta_1 \beta_v^1 + \beta_v^2$. Note that the last sum is modulo $2$. The inequality follows for the subadditivity of the norm $| \cdot |$. The last equality follows from the fact that every triangle of $X$ is inside $q+1$ pyramids and hence $\sum_{v \in X(0)}|\beta_v^2| = (q+1) |\beta|$.
\end{proof}

We are now ready to complete the proof of Theorem~\ref{thm-isoperimetric -inequalities}. Recall that by Proposition~\ref{prop-S(4,q)} there exists $0<\epsilon (4)$ such that for every minimal $1$-cochain $\varphi$ of $S(4,q)$, $||\delta_1(\varphi)|| \geq 3\epsilon(4)||\varphi||$. Now, since every edge in $S(4,q)$ is contained in $q+1$ triangles, it follows that $|\delta_1(\varphi)| \geq \epsilon(4)(q+1)|\varphi|$.

Write our $\alpha$ as $\gamma_0 + \gamma_1 + \gamma_2 + \gamma_3$ where $\gamma_i$ is the subset of all the triangles of $\alpha$ which have exactly
$i$ thick vertices.  Assume first that $|\gamma_0| > \frac{\epsilon(4)}{100}|\alpha |$. This means that $\frac{\epsilon(4)}{100}$ fraction of the triangles of $\alpha$ have no thick vertex. For such a triangle, almost surely, (when $q \gg 0$), at least $2$ edges are thin (Lemma~\ref{lemma-neighborhood-of-thin-vertex}(\ref{item-two-lemma-neighborhood-of-thin-vertex})). In addition, we know that almost every triangle of $\alpha$ has at least one thin edge (Lemma~\ref{lemma-neighborhood-of-thin-vertex}(\ref{item-one-lemma-neighborhood-of-thin-vertex})). Thus, $r > (1 +  \frac{\epsilon(4)}{100}-o_q(1))|\alpha |$ and Lemma~\ref{lemma-t-1-large-if-many-thins} now finishes the proof.

Assume therefore that $| \gamma_0 | <  \frac{\epsilon(4)}{100}|\alpha|$. Note that by Corollary~\ref{cor-almost-every-alpha-triangle-has-at-most-one-thick-vertex}, $|\gamma_2 + \gamma_3 | \leq \frac{20}{q^{0.1}} |\alpha|$.

Thus,
$$|\gamma_1| = |\alpha -\gamma_0 -\gamma_2 -\gamma_3 | \geq |\alpha|- \frac{\epsilon(4)}{100}|\alpha| - \frac{20}{q^{0.1}} |\alpha| \geq (1 - \frac{\epsilon(4)}{50})|\alpha|, \mbox{ for }q \gg 0.$$
On the other hand, as in Proposition~\ref{prop-b-v-1-and-b-v-2},
$$|\delta_2 (\gamma_1)| \geq \frac{1}{4}(\sum_{v \in X(0)}|\delta_1 \gamma_{1,v}^1 + \gamma_{1,v}^2 | ) \geq
\frac{1}{4}(\sum_{v \in S_0}|\delta_1 \gamma_{1,v}^1 + \gamma_{1,v}^2 | ) \geq
\frac{1}{4}(\sum_{v \in S_0}|\delta_1 \gamma_{1,v}^1 |) -\frac{1}{4}\sum_{v \in S_0} |\gamma_{1,v}^2 |, $$
where $S_0$ is the set of thick vertices.

Now note: by our assumption, $\alpha$ is locally minimal, i.e.,  $\alpha_v^1$ is minimal (i.e. closest possible to $B^1(X_v,\F_2)$ in its coset), the same is true for $\gamma_{1,v}^1$, as the later is a subset of $\alpha_v^1$ and hence also minimal. Thus, $$\sum_{v \in S_0}|\delta_1 \gamma_{1,v}^1 | \geq \sum_{v \in S_0}\epsilon(4)(q+1)|\gamma_{1,v}^1 | =\epsilon(4)(q+1)\sum_{v \in S_0}|\gamma_{1,v}^1 | = \epsilon(4)(q+1)|\gamma_{1} |.$$ The last equality is true since every triangle of $\gamma_1$ has a unique thick vertex.

Let us now evaluate $\sum_{v \in S_0} |\gamma_{1,v}^2 |$. Note that $\gamma_{1,v}^2 $ gives $1$ only to pyramids containing $v$ as well as another thick vertex, and just one like that. Thus,  $\sum_{v \in S_0} |\gamma_{1,v}^2 |$ is bounded by twice the number of pyramids with two thick vertices.
Recall that by Lemma~\ref{lemma-number-of-edges-within-S-0}, $E(S_0,S_0) < \frac{20}{q^{2.1}}|\alpha|$.
Every pyramids with two thick vertices contains an edge from $E(S_0,S_0)$. On such an edge there are at most $2(q^2+q+1)$ triangles and on each triangle $(q+1)$ pyramids. This implies that
$$\sum_{v \in S_0} |\gamma_{1,v}^2 | < 2\frac{20}{q^{2.1}}|\alpha| \cdot 2(q^2+q+1) \cdot (q+1) \leq 100 q^{0.9} |\alpha|.$$
Putting the last three inequalities together we get that,
$$|\delta_2(\gamma_1)| \geq \frac{1}{4}\epsilon(4)(q+1)|\gamma_1| - 25q^{0.9}|\alpha|.$$
Finally we can compute:
\begin{eqnarray}
|\delta_2 \alpha | & =    & |\delta_2(\gamma_0 + \gamma_1 + \gamma_2 + \gamma_3)|  \\
                     & \geq & |\delta_2(\gamma_1)| - |\delta_2(\gamma_0)| - |\delta_2(\gamma_2+\gamma_3)| \\
                     & \geq &  \frac{1}{4}\epsilon(4)(q+1)|\gamma_{1} | - 25 q^{0.9} |\alpha| - (q+1)\frac{\epsilon(4)}{100}|\alpha| -(q+1)\frac{20}{q^{0.1}}| \alpha|
\end{eqnarray}
We used here the fact that for every $\beta \in C^2(X,\F_2)$, $|\delta_2 \beta| \leq (q+1) |\beta |$, and also that we are now under the assumption that $|\gamma_0| < \frac{\epsilon(4)}{100}|\alpha|$ and $|\gamma_2 + \gamma_3| < \frac{20}{q^{0.1}}|\alpha|$.

Now, $|\gamma_1| \geq (1- \frac{\epsilon(4)}{50})|\alpha|$  and altogether

$$|\delta_2 \alpha | \geq  \frac{1}{4}\epsilon(4)(q+1)(1- \frac{\epsilon(4)}{50})|\alpha| -  (q+1)\frac{\epsilon(4)}{100}|\alpha| - 25 q^{0.9} |\alpha| -(q+1)\frac{20}{q^{0.1}}| \alpha| \geq \epsilon(4)(q+1)|\alpha|(\frac{1}{4} - \frac{\epsilon(4)}{200} - \frac{1}{100}) - 50 q^{0.9}|\alpha|$$

As $\epsilon(4)(\frac{1}{4}-\frac{\epsilon(4)}{200}-\frac{1}{100}) \geq 0.2 \epsilon(4)$ we get:

$$|\delta_2 \alpha | \geq  0.2 \epsilon(4) (q+1)|\alpha| - 50 q^{0.9} |\alpha| \geq 0.1\epsilon(4)(q+1)|\alpha|.$$ for $q$ sufficiently large and Theorem~\ref{thm-isoperimetric -inequalities} is proved.

%Note again that Theorem~\ref{thm-isoperimetric -inequalities} is weaker than its analogue Theorem~\ref{thm-two-dim-coboundary-exp}. It gives the result for $|\alpha| \leq q^3n$, but $|\alpha|$ can be as large as $\Theta(q^5 n)$, as there are $\Theta(q^5 n)$ triangles in $X(2)$. So, for a fixed $q$ we get a linear lower bound on the size of non-trivial cocycles, but not one that is independent of $q$ as in Corollary~\ref{cor-systolic-two-dim}.
%\tnote{is it independent of $q$ in Corollary~\ref{cor-systolic-two-dim}?}

\begin{remark} Theorem~\ref{thm-isoperimetric -inequalities} was proved with $\epsilon_0, \epsilon_1, \epsilon_2$ that are absolute constant independent of $q$ (provided $q \gg 0$). So are $\mu_0$ and $\mu_1$, but $\mu_2 = \frac{c''}{q^2}$ depends on $q$. It is interesting and somewhat useful to know if $\mu_2$ can be made to be independent of $q$.
\end{remark}

We have the following systolic corollary:

\begin{cor}\label{cor-systolic-three-dim}
For $q \gg 0$, if $\alpha \in C^2(X,\F_2)$ represents a non-trivial 2-cohomology class than $|\alpha| \geq q^3 |X(0)|$.
\end{cor}

\section{Coboundary expanders and the congruence subgroup property}
\label{section-Serre-conj-implications}

Theorem~\ref{thm-Gromov-criteria-for-top-exp} implies that coboundary expanders are topological expanders. Theorem~\ref{thm-main} gives, for $d=2$, a family $\{Y_{a}\}$ of topological expanders. But, our family falls short of being a family of coboundary expanders. In fact, Proposition~\ref{prop-non-trivial-first-second-cohomology} shows that for many of them $H^1(Y_{a},\F_2) \neq 0$ which violates the expansion property. On the other hand, our proofs of Theorem~\ref{thm-main} shows that this is the only obstacle i.e., these complexes that we construct for the proof of Theorem~\ref{thm-main} (i.e., the $2$-skeletons of $3$-dimensional non-partite Ramanujan complexes) would be coboundary expanders if their first cohomology over $\F_2$ would vanish. Indeed, we prove their cocycle expansion; If $H^1 =0$, this is the same as coboundary expansion for $Y_a$. The goal of this short section is to explain that, assuming a very special case of a conjecture of Serre~\cite{Serre}, infinitely many of the examples we constructed have indeed vanishing $1$-cohomology. Hence they form an infinite family of bounded degree $2$-dimensional coboundary expanders.

Let us recall the following standard definitions: Let $k$ be a global field, i.e., a finite extension of $\Q$ or a field of transcendental degree $1$ over $\F_q$ in the positive characteristic case, $\calO$ the ring of integers of $k$, $S$ a finite set of valuations of $k$ including all the archimedean ones, and $\calO_S=\{x \in k | \nu(x) \geq 0, \forall\nu \notin S\}$-the ring of S-integers. Let $G$ be a simply connected, connected, simple algebraic group defined over $k$ with a $k$-embedding $G\hookrightarrow GL_n$ and $G(\calO_S)=G \bigcap GL_n(\calO_S)$. We say that $G(\calO_S)$ has the {\em congruence subgroup property} if the "congruence kernel": $C(G,S)=\mbox{Ker}(\widehat {G(\calO_S)} \rightarrow G(\widehat{\calO_S} ))$ is finite, where $(\widehat{\mbox{ . }})$ denotes the propfinite completion. The reader is referred to~\cite{Serre}~\cite{PrasadRapinchuk},~\cite{Ragunathan} for details and history of this problem. Serre~\cite{Serre} conjectured that this is indeed the case if $S\mbox{-rank}(G) \geq 2$, i.e., if $\sum_{v \in S}k_v\mbox{-rank}(G) \geq 2$.

We are interested in the Cartwright-Steger arithmetic lattice - the CS-lattice. This discrete co-compact subgroup of $PGL_d(\F_q((t)))$ is an arithmetic group (see~\cite{CS} and ~\cite{LSV2} for a detailed construction and~\cite{LubotzkyJapan} for an exposition) and its S-rank is $d-1$. In particular, the case used in this paper for Theorem~\ref{thm-main} is $d=4$, so the S-rank equals $3$, and it is covered by Serre's conjecture. While Serre's conjecture has been proven for "most" cases, the special case of the CS-lattices is still open (see~\cite{PrasadRapinchuk},~\cite{Ragunathan} for a survey of the current knowledge). Let us say first, that assuming Serre's conjecture for the CS-lattice, the congruence kernel is not just finite, but actually trivial. Indeed, in this case the Margulis-Platonov conjecture (on the finite index subgroups of $G(k)$) is known to hold (\cite{RapinchukSegev},\cite{Ragunathan}) and as explained in~\cite{PrasadRapinchuk} if $C(G,S)$ is finite, it is also central and isomorphic to the metaplectic kernel of $G$ with respect to $k$ and $S$. This metaplectic kernel has been computed by Prasad and Rapinchuk~\cite{PrasadRapinchuk1} and in our case is trivial.

Assume now that the Serre's Conjecture indeed holds in our case, where $d=4$ and $q$ a fixed large odd prime power $q=p^r$. Then for the CS-lattice $\Gamma=G(\calO_S)$, $\widehat{G(\calO_S)}=G(\widehat{\calO_S})$. The simplicial complexes used for the proof of Theorem~\ref{thm-main} are $X_{a}=\Gamma_{a}\backslash B$ when $B=A_3(\F_q((t)))$ and where $\Gamma_{a}$ are congruence subgroups of $\Gamma$, and $X_{a}$ are non-partite. As explained in~\cite{LSV2}, $\Gamma_{a}$ can be taken to be a principle congruence subgroup of the form $\Gamma(I)=\mbox{Ker}(G(\calO_S) \rightarrow G(\calO_S/I))$ when $I \triangleleft \calO_S$. Moreover Theorem $7.1$ there ensures that for every sufficiently large $e$ there exists an irreducible polynomial $f(x)$ of degree $e$, so that if $I=(f(x))$ is the ideal generated by $f(x)$, $X_I = \Gamma(I)\backslash B$ is non-partite. Now, the assumption that $\widehat{G(\calO_S)} = G(\widehat{\calO_S})$ implies that for such an $I$, $\widehat{\Gamma(I)} = K_{v_f} \times \prod_{v \neq v_f}G(\calO_{S,v})$, where for a valuation $v$ of $k$, $\calO_{S,v}$ is the $v$-completion of $\calO_{S}$, $v_f$ is the valuation associated with $f(x)$ and $K_{v_f}$ is the normal subgroup of $G(\calO_{S,v_f})$,
$K_{v_f} = \mbox{Ker}(G(\calO_{S,v_f}) \rightarrow G(\calO_{S,v_f}/f(x)\calO_{S,v_f}))$. The group $K_{v_f}$ is a pro-p group and each of $G(\calO_{S,v})$ is an extension of a pro-p group by a quasi-simple finite group. As $p$ is odd, it follows that $[\widehat{\Gamma(I)}, \widehat{\Gamma(I)}]\widehat{\Gamma(I)}^2=\widehat{\Gamma(I)}$, i.e., $\widehat{\Gamma(I)}$ has no quotient that is abelian of order $2$. The same holds therefore for $\Gamma(I)$. Hence, $H^1(\Gamma(I),\F_2)=0$. Now, as the building $B$ is contractible, it follows that $H^1(X_I,\F_2)=0$. Finally, $Y_I$, the 2-skeleton of $X_I$, satisfies $H^1(Y_I,\F_2) = H^1(X_I,\F_2)$ and so $H^1(Y_I,\F_2)=0$ as promised. We can summarize:

\begin{cor} Assume that for some large odd prime power $q$, the Cartwright-Steger arithmetic lattice of $PGL_4(\F_q((t)))$ satisfies the congruence subgroup property (as predicted by Serre's conjecture). Then there exists an infinite family of bounded degree $2$-dimensional coboundary expanders.
\end{cor}

\begin{remark} In fact, a much weaker assumption than Serre's Conjecture is needed. Namely, Serre predicts that $C=C(G,S)$ is trivial (in our case). We only need that $C/[C,C]C^2$ is trivial.
\end{remark}

\paragraph{Acknowledgments.}
The authors are grateful to G. Kalai, N. Linial, R. Meshulam, E. Mossel, S. Mozes, A. Rapinchuk, J. Solomon and U. Wagner for useful discussions and advice. We thank also the ERC, ISF, BSF and NSF for their support.

\end{document}